\documentclass[a4paper,oneside,11pt]{article}
\usepackage{amsmath}
\usepackage{amssymb}
\usepackage{titlesec}
\usepackage{graphicx}
\usepackage{fancyhdr}
\usepackage{latexsym}
\usepackage{mathrsfs}
\usepackage{booktabs}
\usepackage{mathrsfs}
\usepackage{mathtools}
\usepackage{amsthm}
\usepackage{wasysym}
\usepackage{cite}
\usepackage{color}
\usepackage{amsfonts}
\usepackage{pgf,pgfarrows,pgfnodes,pgfautomata,pgfheaps}
\usepackage{amsmath,amssymb}
\usepackage{graphics}
\usepackage{multimedia}
\usepackage{graphics}
\usepackage{cite}
\usepackage{latexsym}
\usepackage{amsmath}
\usepackage{amssymb}
\usepackage[dvips]{epsfig}
\usepackage{amscd}
\numberwithin{equation}{section}
\newtheorem{definition}{Definition}[section]
\newtheorem{theorem}{Theorem}[section]
\newtheorem{lemma}{Lemma}[section]
\newtheorem{remark}{Remark}[section]

\newcommand{\thatsall}{\hfill$\Box$}
\newcommand{\R}{{\mathbb{R}^2}}

\date{}
\title{Hardy Space and Exponential Decay of Inhomogeneous Navier-Stokes System in $\mathbb{R}^2$ }
\author{Xinyu F{\small an}$^{a} $, Boqiang L{\small \"U}$^{b} $, Xue W{\small ANG}$^{c}$  \thanks{Email: fanxinyu17@mails.ucas.ac.cn(X. Fan), lvbq86@ncu.edu.cn \& lvbq86@163.com (B.   L\"u), xuewa@amss.ac.cn (X.  Wang) } \\
{\normalsize a. School of Mathematical  Sciences, } \\
{\normalsize  Nankai University, Tianjin 300110, P. R. China;}
 \\ {\normalsize b. School of Mathematics and Computer Sciences, }\\ {\normalsize \& Institute of Mathematics and Interdisciplinary Sciences,}\\ {\normalsize  Nanchang University, Nanchang 330031, P. R. China;}\\
{\normalsize c. School of Mathematical Sciences,} \\ {\normalsize    Tiangong University, Tianjin 300380, P. R. China.}}

\begin{document}
\maketitle
\begin{abstract}
We investigate the inhomogeneous incompressible Navier-Stokes system in $\R$ with 
zero initial total momentum, 
allowing vacuum and densities with finite moments but not necessarily compact support. Exploiting the cancellation structure inherent in the 2D equations along with spatial moment bounds on the density, 
we derive that $\rho\dot{u}$ belongs to the Hardy space $\mathcal{H}^1(\R)$ and establish weighted Stokes estimates through endpoint Hardy-Littlewood-Sobolev inequalities.
These estimates identify the far-field velocity $u_\infty(t)$ and control the remainder $u-u_\infty$ in \(L^p(\mathbb R^2)\), leading to uniform exponential decay of $u$ throughout $\R$. Combined with higher-order regularity bounds,
we establish  the global well-posedness of strong solutions for the density-dependent viscosity case, and prove the  
uniform regularity of flow maps for measurable density patches in the constant viscosity model.

\end{abstract}

\section{Introduction}
\quad The evolution of an inhomogeneous incompressible viscous flow in $\mathbb{R}^2$ is described by the following Navier-Stokes equations: 
\begin{equation}\label{C101}
\begin{cases}
\rho_t+\mathrm{div}(\rho u)=0,\\
(\rho u)_t+\mathrm{div}(\rho u \otimes u)-\mathrm{div}(2\mu(\rho)Du)+\nabla P=0,\\
\mathrm{div}u=0.
\end{cases}
\end{equation}
Here, $t\ge 0$  stands for the time variable and  $x=(x^1,x^2)\in\mathbb{R}^2$ for the spatial coordinate. The unknowns  $\rho$, $u=(u^1,u^2)$, and $P$ represent the fluid density, velocity and pressure respectively.  $Du=\frac12\big( \nabla u+(\nabla u)^T\big)$ is the deformation tensor. The viscosity  $\mu =\mu(\rho)$ depends on $\rho$ and satisfies that 
\begin{equation} \label{mu}
\mu(s) \in C^1[0, \infty), ~~   ~~  \mu(s)>0 ~ \text{for  any }\,s\in[0,\infty).
\end{equation}

We look for the solution $(\rho,u)$ to the Cauchy problem for \eqref{C101} with the density vanishing at infinity and  supplemented with the  initial conditions:
\begin{equation}\label{1.2}
\rho(x,0)=\rho_0(x),\quad \rho u(x,0)=\rho_0 u_0(x),\quad x\in\R.
\end{equation} 

The classical existence theory of strong and weak solutions to inhomogeneous incompressible equations was developed by Ladyzhenskaya-Solonnikov \cite{LS} and Kazhikhov \cite{K} respectively, for the initial density $\rho_0$ is strictly positive. 
Then Simon \cite{Sm} removed the restriction on the lower bound of $\rho_0$ and obtained global weak solutions with a  finite energy, while Lions \cite{lions} further investigated the system with the viscosity depending on the density. 
Later, for densities close to a positive constant, Danchin \cite{Dan} obtained the well-posedness of \eqref{C101} in the critical functional framework, followed by further developments in Besov spaces by Abidi-Paicu \cite{AP} and Paicu-Zhang \cite{PZ}, see also \cite{AGZ,HPZ} and references therein. For flows allowing vacuum, Craig-Huang-Wang \cite{CHZ} generalized the result of Fujita-Kato \cite{FK} to \eqref{C101} in $\mathbb{R}^3$. 
Huang-Wang \cite{hw1} and L\"{u}-Shi-Zhong \cite{lsz} also constructed global strong solutions to \eqref{C101} in 2D bounded domains and $\R$ respectively. 

Beyond  existence theory, a further central challenge lies in establishing  sufficient velocity regularity under weak assumptions on the density, especially when discontinuities or vacuum are present. 
This issue is particularly illustrated 
in Lions' density patch problem, for which the initial density is given by a characteristic function of some domain $D$, say $\rho_0=\mathbf{1}_D$. The regularity of velocity plays a crucial role in the propagation of interface regularity (see \cite{C,BP}).
Liao-Zhang \cite{LZ} first obtained the global regularity of 2D density patches by assuming $\rho_0=\alpha{\bf 1}_D+\beta{\bf 1}_{D^c}$, where  $\alpha,\beta>0$ and $D$ is a simply connected $W^{3,p}$ domain. The breakthrough due to Danchin-Mucha \cite{DM,DM2} provided a class of strong solutions for density patches containing vacuum in the torus and a connected bounded domain. Then Prange-Tan \cite{PT} and Hao-Shao-Wei-Zhang \cite{HSWZ, HSWZ2} extended the related results to the 2D Cauchy problem with constant viscosities, by establishing proper interpolation and weighted estimates. 
If viscosities depend on density, it's more difficult to derive 
the estimation of  $\|\nabla u\|_{L^\infty}$
due to the low regularity of density. By imposing the striated regularity on $\mu(\rho_0)$, Paicu-Zhang \cite{PZ2} obtained the  global well-posedness of \eqref{C101} for non-vacuum initial data. Gancedo, Garc\'{i}a-Ju\'{a}rez, and Luna-Velasco \cite{GGL} 
established global $C^{1,\gamma}$ interface regularity with small jumps, including nonnegative densities in the torus. See also \cite{CSV,Dan2,DW} for advances in this direction.

The well-posedness theory \cite{LZ,HSWZ,DM2,PT} provides propagation of regularity on every finite time interval.
However, the uniform-in-time control and  limiting dynamics require further detailed estimates, particularly on the velocity decay. In the bounded domains and torus, Danchin-Mucha-Piasecki \cite{DMP} and Gancedo {\it et al.} \cite{GGL} obtained the exponential stability of the system \eqref{C101}. For 3D Cauchy problems of \eqref{C101}, 
Abidi-Gui-Zhang \cite{AGZ2} derived the algebraic decay speed of solutions when $\rho$ is away from vacuum, and He-Li-L\"{u} \cite{hll} showed that  the solutions  decay at exponential rates if  $\rho$ tends to $0$ at the far field. 
Recently, \v{S}kondri\'{c}-Violini \cite{SV} established the global well-posedness and exponential decay for bounded Lipschitz density patches surrounded by vacuum, with initial velocities in the critical space $\dot B^0_{2,1}(\mathbb R^2)$. More related issues can be found in \cite{PZ2,GGL,lx,H} and references therein.

These developments motivate a quantitative description of the velocity throughout the whole space and its long-time evolution in \eqref{C101}.
One of the major difficulties in 2D Cauchy problems is caused by the vacuum and endpoint Sobolev inequalities. In fact, the basic energy estimates only provide that 
$$\sqrt{\rho}u,\,\nabla u\in L^2(\R).$$
In the presence of vacuum, the critical Sobolev embedding only provides $u\in BMO$, which does not ensure the global $L^p$ integrability of $u$, even after subtracting a constant.
However, we show that the cancellation structure inherent in the 2D equations (see also \cite{SV}), 
along with the spatial moment bound of $\rho$, allows us to decompose 
\begin{equation*} 
\begin{split}
u=u_\infty(t)+v,~~v\in L^p(\R), 
\end{split}
\end{equation*}
with $u_\infty(t)$ is the far-field constant state of $u$.
This decomposition is achieved with the help of Hardy space $\mathcal{H}^1(\R)$,  the natural tool to capture cancellation structures and deal with endpoint estimates. 
Combined with energy and higher-order estimates, these velocity bounds yield exponential decay of solutions and the following results: 
\begin{itemize}

\item Global well-posedness of strong solutions allowing vacuum for the model with density-dependent viscosity.

\item Uniform-in-time regularity and exponential convergence of flow maps for bounded density with finite moments in the constant-viscosity case, which includes measurable density patches.
\end{itemize}

Let us introduce the notation in this paper.
For the integer $k\geq 0$ and $1 \leq p \le \infty$, the  standard Lebesgue spaces, H\"{o}lder spaces, and Sobolev spaces are denoted by:
\begin{equation*}\begin{cases}
L^p\triangleq L^p(\mathbb{R}^2),\quad C^\alpha\triangleq C^\alpha(\R),\quad  W^{k, p}\triangleq W^{k, p}(\mathbb{R}^2),\quad H^k\triangleq W^{k, 2},\\
D^{k,p}\triangleq D^{k,p}(\mathbb{R}^2)=\{v\in L^1_{\mathrm{loc}}(\mathbb{R}^2)\,\big|\,\nabla^k v\in L^p\},\quad D^k\triangleq D^{k,2},\\
C_{0,0}^\infty=\big\{w\in C_0^\infty(\mathbb{R}^2)\,\big|\,\mathrm{div}w=0\big\},\quad D^1_{0,0}= \mathrm{the~closure~of}~{C_{0,0}^\infty}~\mathrm{in}~D^1.
\end{cases}
\end{equation*}
The material derivative of a function $f$ refers to $\dot{f}\triangleq\partial_tf+u\cdot\nabla f.$ 

Throughout the paper, weak solutions satisfy the equations in the sense of distributions, with the initial conditions understood in the corresponding trace sense. Strong solutions have sufficient regularity for the equations to hold almost everywhere, while the regularity classes are specified below.
The local existence of strong solutions to \eqref{C101}--\eqref{1.2} is standard in view of \cite{CHZ,lsz}, 
and we directly state the result in  following Lemma \ref{L41} and omit the proof.
\begin{lemma}\label{L41}
For given constants $\bar\rho>0$ and $q>2$, suppose that the initial values $(\rho_0, u_0)$ satisfy
\begin{equation}\label{1.3}
\begin{split}
0\leq \rho_0\leq &\rho,\quad \rho_0\,(1+|x|^2) \in L^1,\quad \rho_0\in H^1\cap W^{1, q},\quad \sqrt{\rho_0}u_0\in L^2,\quad u_0\in D^1_{0,0}.
\end{split}
\end{equation}
Then there exists a time $T>0$ determined by $\bar\rho$, $\mu$, $q$, and the initial data, such that we can construct a unique strong solution $(\rho,u,P)$ to the system \eqref{C101}--\eqref{1.2} on $\R\times(0,T]$. 
Moreover, the following regularity assertions hold for every $0<\tau<T$: 
\begin{equation}\begin{cases}\label{1.7} 
\rho\in  C([0,T]; H^1\cap W^{1, q}),\quad
\rho\,(1+|x|^2)\in L^\infty(0,T;L^1),\\
\sqrt\rho u,\nabla u\in L^\infty(0,T;L^2),\\
\sqrt\rho u_t\in L^2(0,T;L^2)\cap L^\infty(\tau,T;L^2),\quad \nabla u_t\in L^2(\tau,T;L^2),\\
\nabla u,P\in L^2(0,T;H^1)\cap L^\infty(\tau,T;H^1)\cap L^2(\tau,T;W^{1,q}).
\end{cases}\end{equation}
\end{lemma}

Throughout the rest of the paper, we assume that the initial total momentum is zero for simplicity, 
\begin{equation}\label{zero}
\int_\R \rho_0 u_0(x)\,dx=0,
\end{equation}
while the general case can be handled by a Galilean transformation.
Now we give the crucial technical theorem describing the far field state $u_\infty$ of the velocity  $u$ and the integrable remainder $v=u-u_\infty$.
\begin{theorem}\label{T3}
Let $(\rho,u,P)$ be the strong solution provided by Lemma \ref{L41}. 
Then for any $\beta\in(0,1)$, there is a constant $\varepsilon$ depending on $\beta$,
such that if
\begin{equation}\label{CC90}
  \|\mu(\rho)-1\|_{L^\infty}\le \varepsilon,
\end{equation}
then the following conclusions hold.

1. There exists a unique far field velocity $u_\infty(t)$ such that
$$u(x,t)=u_\infty(t)+v(x,t),~~v\in L^{4/\beta}(\R).$$
Moreover, $u(x,t)\rightarrow u_\infty(t)$ as $|x|\rightarrow\infty$.

2. For $m=\int_\R\rho\,dx$, the far field velocity $u_\infty(t)$ is given by
\begin{equation}\label{CC99}
u_\infty(t)=\frac{1}{m}\int_\R\nabla\Delta^{-1}\rho\cdot\nabla u\,dx.
\end{equation}
This integral converges absolutely, and there is a positive constant $C_0$ such that
\begin{equation}\label{CC100}
|u_\infty(t)|\leq C_0\,\|\nabla u\cdot(1+|x|)^{\beta/2}\|_{L^2}.
\end{equation}

3. For the remainder $u-u_\infty(t)$, there is a positive constant $C_1$ such that 
\begin{equation}\label{CC101}
\|u-u_\infty(t)\|_{L^{4/\beta}}\leq C_1\|\nabla u\cdot|x|^{\beta/2}\|_{L^2}.
\end{equation}
Set $M(t)=\int_\R\rho|x|^2\,dx$. Then there is a constant $C_2$ such that
\begin{equation}\label{CC102}
\|\nabla u\cdot|x|^{\beta/2}\|_{L^2}\leq
 C_2 \big(1+M(t)^{\beta/2}\big)\,\|\sqrt{\rho}\dot u\|_{L^2}.
\end{equation}

The constants $C_0$, $C_1$, and $C_2$ depend only on $m$, $\bar\rho$, $\mu$, and $\beta$, but not on the terminal time $T$.
\end{theorem}

Applying the velocity representation theorem, we establish global well-posedness and exponential decay of strong solutions for the density-dependent viscosity model.
\begin{theorem}\label{T1}
Let $(\rho_0,u_0)$ satisfy the condition in Lemma \ref{L41} and \eqref{zero}. There exists a positive constant $\tilde\varepsilon$ determined by $\bar\rho$, $\mu$, $q$, and the initial values 
such that if  
\begin{equation*}
\|\mu(\rho_0)-1\|_{L^\infty}\le \tilde\varepsilon,
\end{equation*}
then the local strong solution $(\rho,u,P)$ provided by Lemma \ref{L41}  extends uniquely to a global strong solution on $\R\times(0,\infty)$ satisfying \eqref{1.7} on every finite time interval, for all $0<\tau<T<\infty$.
Moreover, the following conclusions hold.

1. $($Uniform bounds on $\rho)$  There is a positive constant $\mathbf{C}_1$ such that
\begin{equation}\label{C102}
\sup_{t\geq 0}\|\rho(\cdot,t)\|_{L^\infty}\leq\bar\rho,\quad \sup_{t\geq 0}\left(\big\|\rho(\cdot,t)\,(1+|x|^2)\big\|_{L^1}+\|\rho(\cdot,t)\|_{ H^1\cap W^{1,q}}\right)\leq\mathbf{C}_1.
\end{equation}

2. $($Energy and derivative decay$)$ There are positive constants  $\mathbf{C}_2$ and $\sigma$ such that
\begin{equation}\label{1.8}
\|\sqrt\rho u(\cdot,t)\|_{L^2}+\|\nabla u(\cdot,t)\|_{H^1}+\|P(\cdot,t)\|_{H^{1}}
\leq \mathbf{C}_2\exp\left\{-\sigma t \right\},~~\forall t\geq 1.
\end{equation}

3. $($Whole-space velocity decay$)$ Let $u_\infty$ be the far field velocity in Theorem \ref{T3}. For any $p\in[5,\infty]$, there are positive constants  $\mathbf{C}_3$ and $\sigma$  such that
\begin{equation}\label{CC97}
|u_\infty(t)|+\|u(\cdot,t)-u_\infty(t)\|_{L^p}\leq \mathbf{C}_3\exp\left\{-\sigma t \right\},~~\forall t\geq 1,
\end{equation}
In particular,
\begin{equation}\label{CC96}
\|u\|_{L^\infty}\leq \mathbf{C}_3\exp\left\{-\sigma t \right\},~~\forall t\geq 1,
\end{equation}
therefore the velocity converges exponentially to $0$ uniformly throughout $\R$, including the vacuum region.

The constants $\mathbf{C}_1$, $\mathbf{C}_2$, and $\sigma$ depend on $\bar\rho$, $\mu$, $q$, and initial data, while $\mathbf{C}_3$ additionally depends on $p$.
None of these constants depend on the terminal time $T$.
\end{theorem}

We next consider weak solutions of the system \eqref{C101}--\eqref{1.2}. If the initial density is bounded without spatial regularity assumptions, then for the constant-viscosity case, Theorem \ref{T3} extends to this setting and yields whole-space exponential decay along with uniform-in-time regularity of the flow map. 
We mention that density patches are an important special case and the global well-posedness of weak solutions has already been established in \cite{DM2,lsz,HSWZ}.
\begin{theorem}\label{T2}
For the constant $\bar\rho>0$, suppose that the initial values $(\rho_0, u_0)$ satisfy
\begin{equation}\label{C103}
\begin{split}
0\leq \rho_0\leq \bar\rho,~~ \rho_0\,(1+|x|^2) \in L^1,~~ \sqrt{\rho_0}u_0\in L^2,~~ u_0\in L^2\cap D^1_{0,0},
\end{split}
\end{equation}
and \eqref{zero}. 
If $\mu(\rho)\equiv 1$, then
the system \eqref{C101}--\eqref{1.2} admits a unique global weak solution $(\rho, u,P)$ and Theorem \ref{T3} holds for $u$. For any $0<\tau<T<\infty$, $(u,P)$ satisfy \eqref{1.7}, but the regularity of density reduces to $\rho\in L^\infty(0,T; L^\infty)$ and $
\rho\,(1+|x|^2)\in  L^\infty(0,T;L^1)$.
Moreover, the following conclusions hold.

1. $($Uniform bounds on $\rho)$  There is a positive constant $\mathbf{C}_1$ such that
\begin{equation}\label{C105}
\sup_{t\geq 0}\|\rho(\cdot,t)\|_{L^\infty}\leq\bar\rho,\quad \sup_{t\geq 0}\big\|\rho(\cdot,t)\,(1+|x|^2)\big\|_{L^1}\leq\mathbf{C}_1. 
\end{equation}

2. $($Energy and derivative decay$)$ There are positive constants $\mathbf{C}_2$ and $\sigma$ such that
\begin{equation}\label{C106}
\|\sqrt\rho u(\cdot,t)\|_{L^2}+\|\nabla u(\cdot,t)\|_{H^1}+\| P(\cdot,t)\|_{H^{1}}
\leq \mathbf{C}_2\exp\left\{-\sigma t\right\},~~\forall t\geq 1.
\end{equation}

3. $($Whole-space velocity decay$)$ Let $u_\infty$ be the far field velocity in Theorem \ref{T3}. For any $p\in[5,\infty]$, there are positive constants  $\mathbf{C}_3$ and $\sigma$  such that
\begin{equation}\label{CC98}
|u_\infty(t)|+\|u(\cdot,t)-u_\infty(t)\|_{L^p}\leq \mathbf{C}_3\exp\left\{-\sigma t \right\},~~\forall t\geq 1.
\end{equation}
In particular,
\begin{equation}\label{CC95}
\|u\|_{L^\infty}\leq \mathbf{C}_3\exp\left\{-\sigma t \right\},~~\forall t\geq 1,
\end{equation}
therefore the velocity converges exponentially to $0$ uniformly throughout $\R$, including the vacuum region.

4. $($Uniform regularity and convergence of flow maps$)$ Let $X(t,x)$ be the flow map defined by
\begin{equation}\label{C109}
\partial_tX_t(x)=u(X_t(x),t),~~
X_0(x)=x.
\end{equation}
For any $\gamma\in(0,1)$, there is a positive constant $\mathbf{C}_4$ such that for any  $T>0$,
\begin{equation}\label{C107}
\begin{split}
\int_0^T\|u(\cdot,t)\|_{L^\infty}\,dt\leq\mathbf{C}_4,\quad
\sup_{0\leq t\leq T}\big\|\nabla X_t(\cdot,t)\big\|_{C^{\gamma}}\leq\mathbf{C}_4,
\end{split}
\end{equation} 
moreover the flow map is uniformly bi-Lipschitz,
\begin{equation}\label{C118}
\mathbf{C}_4^{-1}|x-y|\leq |X_T(x)-X_T(y)|\leq\mathbf{C}_4|x-y|,~~\forall T\geq 0,~x,y\in\R.
\end{equation}

There is a measure preserving $C^{1,\gamma}$ diffeomorphism $X_\infty:\R\rightarrow\R$ given by
\begin{equation*}
X_\infty(x)=\lim_{t\rightarrow\infty}X_t(x)=x+\lim_{t\rightarrow\infty}\int_0^tu(X_s(x),s)\,ds,
\end{equation*}
such that
\begin{equation}\label{CC94}\|X(t,\cdot)-X_\infty(\cdot)\|_{L^\infty}\leq \mathbf{C}_4\exp\left\{-\sigma t \right\},~~\forall t\geq 1.
\end{equation}
Consequently, every trajectory converges to a limiting position, uniformly with respect to its initial position.

The constants $\mathbf{C}_1$, $\mathbf{C}_2$, and $\sigma$ depend on $\bar\rho$,   and initial data, while $\mathbf{C}_3$ and $\mathbf{C}_4$ additionally depend on $p$ and $\gamma$, respectively.
None of these constants depend on the terminal time $T$.
\end{theorem}
\begin{remark}
We now consider density patch problems, with $\rho_0(x)=\mathbf{1}_{D}(x)$. Note that Theorem \ref{T2} requires no regularity assumption on $\partial D$ and holds for measurable density patches.
In particular, if $D$ is a $C^{1,\gamma}$ bounded domain in $\R$, then it holds that 
$\rho(x,t)=\mathbf{1}_{X_t(D)}(x)$
and $\{X_t(D)\}_{t\geq0}$ is a family of domains with uniform $C^{1,\gamma}$ regularity. Moreover, no cusp forms during the dynamic process
$X_t(D)\rightarrow X_\infty(D)$, and the limiting configuration $X_\infty(D)$ is a $C^{1,\gamma}$ domain diffeomorphic to $D$.






\end{remark}


Our proof builds on two cancellation mechanisms inherent to the system\eqref{C101}, along with a resulting damping‑related decay property.

{\it 1. The cancellation of $\rho\dot u$ and weighted Hardy estimates of velocity.}

The Caffarelli-Kohn-Nirenberg inequality reduces the control of the velocity to a weighted estimate, 
\begin{equation}\label{CC105}
  \|u-u_\infty\|_{L^{2/\beta}}\leq C\|\nabla u\cdot|x|^\beta\|_{L^2},~~0<\beta<1,
\end{equation}
therefore the key step is to bound the right hand side of \eqref{CC105} and 
identify the far field velocity $u_\infty$.

However, the weighted estimate of $\nabla u$ in \eqref{CC105} encounters an endpoint obstruction in dimension 2:
Indeed at the level of principal elliptic structure, \eqref{C101}$_2$ is modeled by
\begin{equation}\label{C301}
-\Delta u=f\quad\mathrm{in}~\mathbb{R}^2.
\end{equation}
with $f$ representing the inertial forcing $\rho\dot u$. The standard Hardy-Littlewood-Sobolev inequality does not give weighted $L^2$ estimates of $\nabla u$ from $f\in L^p$ alone. 
In fact, testing against a standard cutoff function $\varphi_R(x)=\varphi(x/R)$, we obtain
$$\left|\int_{\mathbb R^2}f\varphi_R\,dx\right|
\leq CR^{-\beta}
\|\nabla u\cdot|x|^\beta\|_{L^2(B_{2R}\setminus B_R)}.$$
Letting $R\rightarrow\infty$, we find that the cancellation condition $\int_\R f(x)\,dx=0$ is {\bf necessary} to ensure that $\|\nabla u\cdot|x|^\beta\|_{L^2}$ is finite. 

Our main observation is that the above cancellation condition (combined with proper moment conditions)  is also {\bf sufficient} to provide weighted estimates of $\nabla u$.
This is quantified by the weighted Hardy-space estimate (see Section \ref{S32})
\begin{equation}\label{CC110}
\|\nabla u\,|x|^\beta\|_{L^2}\leq C\int_\R\mathcal{M}f\cdot|x|^\beta\,dx\leq C\big(\int_\R|f|\cdot(1+|x|)^\alpha\,dx+\|f\|_{L^p}\big),
\end{equation}
which holds for any $0<\beta<\alpha\leq 1$ and $p>1$ with $\int_\R f\,dx=0$, and $\mathcal{M}f$ denotes the grand maximal function given in Definition \ref{D21}. 
For the momentum equation, the zero-mean condition of $f$ corresponds to the first cancellation mechanism of \eqref{C101}$_2$, $$\int_\R \rho\dot u\,dx=0,$$ 
which is proved by applying a divergence-free cutoff to \eqref{C101}$_2$ (see Section 3.3). 
Such cancellation along with energy estimates and moment bounds yields \eqref{CC105}. 
It remains to identify and estimate the far field velocity $u_\infty$.

{\it 2. The momentum conservation and identification of the far-field velocity.}

Let $u=u_\infty+v$ with $v=u-u_\infty$ be the decomposition obtained above.  
 Multiplying it by $\rho$ and integrating over $\R$ give 
\begin{equation}\label{CC106}
mu_\infty=\int_\R\rho u\,dx-\int_\R\rho v\,dx,~~\mathrm{with}~~m=\int_\R\rho\,dx=\int_\R\rho_0\,dx.
\end{equation}
Testing \eqref{C101}$_2$ with the divergence-free cutoff also yields conservation of total momentum. Together with the zero initial momentum condition \eqref{zero}, this gives the second cancellation mechanism of the system (see also Section 3.3),
\begin{equation}\label{CC107}
\frac{d}{dt}\int_\R\rho u\,dx=0,~~\Rightarrow \int_\R\rho u\,dx=\int_\R\rho_0 u_0\,dx=0.
\end{equation}
Substituting \eqref{CC107} to \eqref{CC106} implies
\begin{equation}\label{CC108}
  u_\infty=\frac{-1}{m}\int_\R\rho v\,dx=\frac{1}{m}\int_\R\nabla\Delta^{-1}\rho\cdot\nabla u\,dx,
\end{equation}
since formally we have $v=\Delta^{-1}(\Delta u)$.
We mention that the integral in \eqref{CC108} converges absolutely due to the control of $\|\nabla u\,|x|^\beta\|_{L^2}$ provided by Hardy type estimates.

{\it 3. A damping mechanism and exponential decay.}

The two cancellation mechanisms and Hardy type estimates provide a damping effect in \eqref{C101}$_2$. 
This structure can be illustrated by a simplified model:
\begin{equation}\label{sm}
\rho(x)\,\partial_tu(x,t)-\Delta u(x,t)=0\quad  \text{in} ~\R\times[0,\infty),
\end{equation}
with a non-negative weight $\rho$ in the Schwartz space.
Taking $L^2$ inner products of \eqref{sm} with $u$ and $u_t$, then adding them together, we obtain the basic energy estimate,
\begin{equation}\label{CC112}
\frac{1}{2}\,\frac{d}{dt}\left(\|\sqrt\rho u\|_{L^2}^2+\|\nabla u\|_{L^2}^2\right)+\left(\|\nabla u\|_{L^2}^2+\|\sqrt{\rho}u_t\|_{L^2}^2\right)=0.
\end{equation}
The conservation of momentum $\int_\R\rho u\,dx=\int_\R\rho_0u_0\,dx=0$ and Hardy type estimates \eqref{CC110} ensure that
\begin{equation}\label{CC111}
\|\sqrt{\rho}u\|_{L^2}^2\leq\|\sqrt{\rho}(u-u_\infty)\|_{L^2}^2\leq C\|u-u_\infty\|_{L^{2/\beta}}^2\leq C\,\|\nabla u\,|x|^\beta\|_{L^2}^2\leq C\,\|\sqrt{\rho}u_t\|_{L^2}^2.
\end{equation}
Note that $\int_\R\rho |u_t|\,(1+|x|)^\alpha\,dx\leq C\|\sqrt\rho u_t\|_{L^2}$ since $\rho$ is rapidly decreasing. Combining \eqref{CC111} with \eqref{CC112}, we therefore obtain the exponential 
decay by Gronwall's inequality. This argument is adapted to the full nonlinear system in Lemma \ref{le5}.

The rest of this paper is organized as follows: In Section \ref{S2}, we collect some basic tools and introduce a specific definition of Hardy space $\mathcal{H}^1(\R)$.  Then we study the weighted $L^2$ estimates for $\nabla u$  in $\R$ with the help of Hardy space in Section \ref{S3} and prove the crucial velocity representation Theorem \ref{T3}. 
After that, we establish all necessary a priori estimates and the exponential decays of solutions in Section \ref{S4}, and finish the proofs of Theorems \ref{T1}--\ref{T2} in Section \ref{S5}.

\section{Analytic Tools}\label{S2}
\quad  
Let $\mathcal{S}$ and $\mathcal{S}'$ denote the spaces of Schwartz functions and tempered distributions in $\R$ respectively,
meanwhile we set $B_N\triangleq \{x\in \mathbb{R}^2\big|\,|x|<N\}$ for $N>0$.
First, we show the well-known Caffarelli-Kohn-Nirenberg inequality to give important weighted estimates in $\R$. Here we cite a part of it, see \cite{ckn,cw} for the full range of indices.
\begin{lemma}\label{pl2}
For any $u\in C^\infty(\mathbb{R}^2)$ and real numbers $\alpha$, $\beta$ satisfying
$$\beta>0,\quad \beta-1\leq\alpha<\beta,\quad 
p=\frac{2}{\beta-\alpha},$$ there exists a constant $u_\infty$ determined by $u$ and a positive  constant $C$ determined only by $\alpha$ and $\beta$ such that
\begin{equation*}
\left(\int_\R|u-u_\infty|^p\,|x|^{p\alpha}\,dx\right)^{1/p}\leq C\left(\int_\R|\nabla u|^2\,|x|^{2\beta}\,dx\right)^{1/2}.
\end{equation*}
\end{lemma}

While for the general weight $\rho$, the next lemma provides corresponding weighted estimates, see \cite{lx} for the proof.
\begin{lemma} 
Suppose that $\rho$ is a non-negative function in $\R$ satisfying
\begin{equation}\notag
\|\rho\|_{L^\infty}\leq M_1,\quad \int_{B_{N}}\rho\, dx\geq M_2,\quad \int_\R\rho\,(1+|x|^2)\,dx\leq M_3, 
\end{equation}
with $N\geq 1$ and $M_1,M_2,M_3\in(0,\infty)$. Then for any $p\in [2,\infty)$, there exists a positive constant $C$ determined only by $M_1$, $M_2$, and $p$ such that for any $u\in C_0^\infty(\R)$,
\begin{equation}\label{2.1}
\left(\int_\R\rho|u|^p\, dx\right)^{1/p}\leq CN^3\left(1+M_3\right)\left(\|\sqrt\rho u\|_{L^2}+\|\nabla u\|_{L^2}\right).
\end{equation}
\end{lemma}

Next, we establish some basic elliptic estimates for the generalized Stokes system.
\begin{lemma}\label{pl3}
Suppose that the function $\mu$ satisfies $\underline{\mu}\leq\mu(x)\leq\overline{\mu}$ for some $\underline{\mu},\,\overline{\mu}\in(0,\infty)$ and any $x\in\R$, we  consider the Stokes-like system in $\R$ defined by
\begin{equation}\label{2.2}\begin{cases}
-\mathrm{div}(2\mu\, Du) +\nabla P=f,\\
\mathrm{div}u=0.\\
\end{cases}
\end{equation} 
If $u$, $P$, $f$ $\in C_0^\infty(\R)$ and satisfy \eqref{2.2}, we have the following declarations.

$1.$ For any $p\in[2,\infty)$, there are positive constants $\varepsilon$ and $C$ depending on 
$\underline{\mu}$, $\overline{\mu}$, and $p$ such that if $\|\mu-1\|_{L^\infty}\leq\varepsilon$, it holds that
\begin{equation}\label{C218}
\|\nabla u\|_{L^p}\leq C\,\|\nabla u\|_{L^2}^{2/p}\|f\|_{L^2}^{1-2/p}.
\end{equation}
 
$2.$
For any $p\in(1,\infty)$, there are positive constants $\varepsilon$ and $C$ depending on  $\underline\mu$, $\overline\mu$, and $p$ such that if $\|\mu-1\|_{L^\infty}\leq\varepsilon$, it holds that
\begin{equation}\label{2.5}
\|\nabla^2 u\|_{L^p}+\|\nabla P\|_{L^p}\leq C\left(\|f\|_{L^p}+\|\nabla\mu\cdot Du\|_{L^p}\right).
\end{equation}
\end{lemma}
\begin{proof}
First, acting the Leray projector $\mathbb{P}\triangleq \mathrm{Id}-\nabla\Delta^{-1}\mathrm{div}$, which is bounded in $L^p$ $(1<p<\infty)$, on \eqref{2.2}$_1$ leads to the following Poisson equation:
\begin{equation}\label{C219}
-\Delta u=\mathbb{P}\,\mathrm{div}\big(2(\mu-1)Du\big)+\mathbb{P}(f).
\end{equation}
In view of the standard elliptic theory \cite[Chapter 9]{GT}, we have
\begin{equation}\label{C334}
\nabla u=
-\nabla\Delta^{-1}\mathbb{P}\,\mathrm{div}\big(2(\mu-1)Du\big)-\nabla\Delta^{-1}\mathbb{P}(f)\triangleq\mathbf{T}_1\big(2(\mu-1)Du\big)+\mathbf{T}_2(f).
\end{equation}
Consequently, the boundedness of operator ``$\nabla\Delta^{-1}\mathbb{P}\,\mathrm{div}$" in $L^p$ $(1<p<\infty)$ implies
\begin{equation}\label{C221}
\|\mathbf{T}_1\big(2(\mu-1)Du\big)\|_{L^p}\leq C\,\|(\mu-1) Du\|_{L^p}\leq\frac{1}{2}\,\|\nabla u\|_{L^p},
\end{equation}
by setting $\|\mu-1\|_{L^\infty}\leq \varepsilon\triangleq (2C)^{-1}$ with $C$ in \eqref{C221}.
While for $\mathbf{T}_2(f)$, note that acting $\mathbf{T}_2=-\nabla\Delta^{-1}\mathbb{P}$ on \eqref{2.2}$_1$ also leads to 
\begin{equation*}
\begin{split}
&\mathbf{T}_2(f)=\nabla\Delta^{-1}\mathbb{P}\,\mathrm{div}\big(2\mu\,Du\big).
\end{split}
\end{equation*}
We apply the Gagliardo-Nirenberg inequality to argue that for any $p\in[2,\infty)$,
\begin{equation}\label{C224}
\begin{split}
\|\mathbf{T}_2(f)\|_{L^p}
&\leq C\,\|\mathbf{T}_2(f)\|_{L^2}^{2/p}\|\nabla\mathbf{T}_2(f)\|_{L^2}^{1-2/p}\\
&= C\,\big\|\nabla\Delta^{-1}\mathbb{P}\,\mathrm{div}\big(2\mu\,Du\big)\big\|_{L^2}^{2/p}\big\|\nabla\mathbf{T}_2(f)\big\|_{L^2}^{1-2/p}\\
&\leq C\,\|\nabla u\|_{L^2}^{2/p}\|f\|_{L^2}^{1-2/p},
\end{split}
\end{equation}
since the operators ``$\nabla\Delta^{-1}\mathbb{P}\,\mathrm{div}$" and ``$\nabla\mathbf{T}_2$" both map $L^2$ to itself (see \cite[Chapter 1]{S}). 
Collecting \eqref{C334}--\eqref{C224}, we thus obtain \eqref{C218}.

Next,  note that \eqref{2.2}$_1$  can be written as
$$-\Delta u+\nabla P=f+2\nabla\mu\cdot Du+(\mu-1)\Delta u,$$
so the standard $L^p$ estimates of the Stokes system yields that for any $p\in(1,\infty)$, 
\begin{equation*}
\|\nabla^2 u\|_{L^p}+\|\nabla P\|_{L^p}\leq C\left(\|f\|_{L^p}+\|\nabla\mu\cdot Du\|_{L^p}+\|\mu-1\|_{L^\infty}\|\nabla^2u\|_{L^p}\right),
\end{equation*}
which leads to \eqref{2.5} by setting $\|\mu-1\|_{L^\infty}\leq \varepsilon\triangleq (2C)^{-1}$ with $C$ given by the above estimate.
The proof of Lemma \ref{pl3} is completed.
\end{proof}

The maximal characterization of the Hardy space $\mathcal{H}^1(\R)$ is due to Fefferman-Stein, which is the key tool of this paper. We adopt a particular version of it, see \cite[Chapter III]{S} for complete theories.
\begin{definition}\label{D21}
For multi-indices $a$ and $b$, the semi-norm $\|\cdot\|_{a,b}$ for any function $\Phi\in\mathcal{S}$ is defined by
$$\|\Phi\|_{a,b}=\sup_{x\in\R}|x^a\nabla^b\Phi(x)|.$$
Let $\mathcal{F}=\{\|\cdot\|_{a,b}\,\big|\,|a|,|b|\leq 3\}$ be a  collection of semi-norms, the subset $\mathcal{S}_{\mathcal{F}}$ of $\mathcal{S},$ which is dominated by $\mathcal{F}$, refers to
\begin{equation}\label{C206}
\mathcal{S}_{\mathcal{F}}=\{\Phi\in\mathcal{S}\,\big|\,\|\Phi\|_{a,b}\leq 1,~\mathrm{for}~\mathrm{all}~\|\cdot\|_{a,b}\in\mathcal{F}\}.
\end{equation}
For any  $\Phi\in\mathcal{S}_{\mathcal{F}}$ with $\Phi_t(x)\triangleq t^{-2}\Phi(x/t)$, the non-tangential maximal operator $M_\Phi$ with respect to a single $\Phi$ and the grand maximal operator $\mathcal{M}$ with respect to the whole collection $\mathcal{S}_\mathcal{F}$ are given by 
\begin{equation}\label{C204}
M_\Phi f(x)\triangleq \sup_{|y|\leq t}|f*\Phi_t(x-y)|,\quad \mathcal{M}f(x)\triangleq\sup_{\Phi\in\mathcal{S}_{\mathcal{F}}}M_\Phi f(x),
\end{equation}
where $f\in\mathcal{S}'$. Then we define the Hardy space $\mathcal{H}^1(\R)$ as
$$\mathcal{H}^1(\R)\triangleq\{f\in\mathcal{S}'\,\big|\,
\mathcal{M}f\in L^1\}\quad \mathrm{with}\quad \|f\|_{\mathcal{H}^1(\R)}\triangleq\int_\R
\mathcal{M}f(x)\,dx.$$
\end{definition}

Moreover, we list some important properties of Hardy space $\mathcal{H}^1(\R)$, see \cite[Chapter III]{S} for the detailed proof.
\begin{lemma}\label{L21}
Let $f$ be a tempered distribution.



$1.$ If $f\in\mathcal{H}^1(\R)$, it must lie in $L^1$ as well and satisfy the cancellation condition, that is, there is a  positive constant $C'$ independent of $f$ such that
\begin{equation*}
\|f\|_{L^1}\leq C'\|f\|_{\mathcal{H}^1(\R)},\quad \int_\R f(x)\,dx=0.
\end{equation*}

$2.$ $\mathcal{H}^1(\R)$ is weakly compact, that is, for any bounded sequence $\{f_n\}_{n=1}^\infty\subset\mathcal{H}^1(\R)$ satisfying $\sup_n\|f_n\|_{\mathcal{H}^1(\R)}\leq C$ for some positive constant $C$, there is a subsequence $\{f_{n_k}\}_{k=1}^\infty$ and $f\in\mathcal{H}^1(\R)$ with $\|f\|_{\mathcal{H}^1(\R)}\leq C$ such that for any $\psi\in\mathcal{S}$,
$$\lim_{k\rightarrow\infty}\int_{\R} f_{n_k}(x)\,\psi(x)\,dx=\int_\R f(x)\,\psi(x)\,dx.$$

$3.$ For $i=1,2$, the Riesz transform $R_i$ (and similar singular integral operators, see \cite[p.71]{S}) maps $\mathcal{H}^1(\R)$ to itself, that is,  there is a positive constant $C''$  such that for any $f\in\mathcal{H}^1(\R)$,
\begin{equation}\label{C202}
\|R_i(f)\|_{L^1}\leq C'\|R_i(f)\|_{\mathcal{H}^1(\R)}\leq C'C''\|f\|_{\mathcal{H}^1(\R)}.
\end{equation}
\end{lemma}

It is rather difficult to confirm a function $f$ lying in $\mathcal{H}^1(\R)$, namely $\mathcal{M}f\in L^1$. According to \cite[Page 92]{S}, for $f\in L^1\cap L^p$ $(p>1)$ satisfying $\int_\R f\,dx=0$, one has $f$ belongs to $\mathcal{H}^1(\R)$, once $\mathrm{supp}f$ is compact. In fact, this restriction on $\mathrm{supp}f$ can be released a little by requiring $f$  to decay at any algebraic rate.
\begin{lemma}\label{L22}
For any function $f\in C_0^\infty(\R)$ satisfying  $\int_\R f\,dx=0$ and real numbers $0\le\alpha<\beta\le 1$, $p>1$, there is a constant $C>0$ depending on $\alpha$, $\beta$, and $p$ such that
\begin{equation}\label{C203}
\int_\R\mathcal{M}f(x)\,(1+|x|)^\alpha\,dx\leq C\left(\int_\R |f(x)|\,(1+|x|)^\beta\,dx+\|f\|_{L^p}\right),
\end{equation}
where $\mathcal{M}$ is the grand maximal operator introduced in Definition \ref{D21}.
\end{lemma}
\begin{remark}\label{re2.1}
The range $0\leq\alpha<1$ seems optimal under the cancellation condition $\int_\R f(x)\,dx=0$, since in general it only ensures $\mathcal{M}f(x)\sim (1+|x|)^{-3}$. However, if we impose a stronger cancellation condition on $f$, for example $\int_\R x^af(x)\,dx=0$ for some  multi-index $0<|a|\leq 1$,  we can extend  \eqref{C203} to $1\leq\alpha<2$ and so on.
\end{remark}
\begin{remark}
Applying the standard approximation procedures to  \eqref{C203}, we can show that any function $f$ satisfying
$$ f\in L^1\cap L^p,~~\int_\R f\,dx=0,~~\int_\R|f|\cdot|x|^\beta\,dx<\infty,$$
for some $p>1$ and $\beta>0$, must belong to $\mathcal{H}^1(\R)$. In addition, $\mathcal{M}f$ also decays at an algebraic rate.
\end{remark}
\begin{proof}
According to \cite[p.57]{S},  there is a generic positive constant $C$  such that for any $\Phi\in\mathcal{S}_\mathcal{F}$ (defined in \eqref{C206}),
\begin{equation}\label{C213}
M_\Phi f(x)\leq CMf(x),
\end{equation}
where $Mf(x)\triangleq\sup_{R>0}\frac{1}{R^2}\int_{\{|y|\leq R\}}|f(x-y)|\,dy$ is the standard maximal function. Thus we also have $\mathcal{M}f(x)\leq CMf(x)$. When $|x|< 1$, it holds that for any $p>1$, 
\begin{equation}\label{C207}
\int_{\{|x|\leq 1\}}\mathcal{M}f(x)\,(1+|x|)^\alpha\,dx\leq
C\,\|\mathcal{M}f\|_{L^p}\leq
C\,\|Mf\|_{L^p}\leq C\,\|f\|_{L^p},
\end{equation}
since $M$ maps $L^p$ to itself, see \cite[p.13]{S}.

When $2^j\le |x|< 2^{j+1}$ with $j\in\mathbb{N}$,  we apply $\int_\R f(x)\,dx=0$ to check that  for any $\Phi\in\mathcal{S}_{\mathcal{F}}$ with $\Phi_t(z)\triangleq t^{-2}\Phi(z/t)$,  $t>0$, and $|y|\leq t$,
\begin{equation}\label{C208}
\begin{split}
&f*\Phi_t(x-y)\\
&=\int_\R \Phi_t(x-y-z)\, f(z)\,dz\\
&=\int_\R\left(\Phi_t(x-y-z)-\Phi_t(x-y)\right) f(z)\,dz\\
&=\int_\R\left(\Phi_t(x-y-z)-\Phi_t(x-y)\right) f(z)\chi_j(z)\,dz\\
&\quad+\int_\R\Phi_t(x-y-z)\, f(z)\left(1-\chi_j(z)\right)\,dz-\int_\R\Phi_t(x-y)\, f(z)\left(1-\chi_j(z)\right)\,dz\\
&\triangleq\mathbf{I}_1+\mathbf{I}_2+\mathbf{I}_3,
\end{split}
\end{equation}
where $\chi_j(x)=\chi(x/2^{j})$ and $\chi(x)\in C_0^\infty(\R)$ is a truncation function satisfying 
\begin{equation}\begin{aligned}\label{chi}
&\chi(x)=
\begin{cases} 
1,\quad |x|\leq 1/4,\\
0,\quad |x|\geq 1/2,
\end{cases}
\\
 0\leq\chi(x)\leq 1,\quad |\nabla^k&\chi(x)|\leq 100,\quad \text{for any} \, x\in\R,~k\in\{1,2,3\} .
\end{aligned}\end{equation}

For $\mathbf{I}_1$, we calculate that 
\begin{equation}\label{C209}\notag
\begin{split}
|\mathbf{I}_1|&=\left|\int_\R\big(\Phi_t(x-y-z)-\Phi_t(x-y)\big)\, f(z)\chi_j(z)\,dz\right|\\
&=\left|\int_\R z\cdot\left(\int_0^1\nabla 
\Phi_t(x-y-s z)\,ds\right)\, \, f(z)\chi_j(z)\,dz\right|\\
&\leq\int_\R \left(\int_0^1\frac{1}{|x-s z|^3} \frac{|x-s z|^3}{t^3}\left|\nabla\Phi\left(\frac{x-s z}{t}-\frac{y}{t}\right)\right|\,ds\right) |z|\,|f(z)|\,\chi_j(z)\,dz\\
&\leq C\int_\R\left(\int_0^1
\frac{1}{|x-sz|^3}\,ds\right)|z|\,|f(z)|\,\chi_j(z)\,dz,
\end{split}
\end{equation}
where we have applied  the fact:
\begin{equation}\notag
\label{C212}\frac{|x-s z|^3}{t^3}\left|\nabla\Phi\left(\frac{x-s z}{t}-\frac{y}{t}\right)\right|\leq C \sup_{w\in\R}\left((|w|^3+1)\,|\nabla\Phi(w)|\right) \leq C,
\end{equation}
since $|y|<t$ and $\Phi\in\mathcal{S}_\mathcal{F}$. Note that $|z|\leq |x|/2$ in the support of $\chi_j(z)$, thus we have
\begin{equation}\label{C210}
\begin{split}
|\mathbf{I}_1|\leq C\int_\R\left(\int_0^1
\frac{1}{|x-sz|^3}\,ds\right)|z|\,|f(z)|\,\chi_j(z)\,dz \leq\frac{C}{|x|^{2+\beta}}\int_\R |f(z)|\,|z|^\beta\,dz.
\end{split}
\end{equation}
Next, we turn to $\mathbf{I}_2$.  Since $|z|\geq |x|/8$ in the support of $1-\chi_j(z)$, we declare that 
\begin{equation}\label{C215}
\begin{split}
|\mathbf{I}_2|&=\left|\int_\R\Phi_t(x-y-z)\, f(z)\left(1-\chi_j(z)\right)\,dz\right|\\
&\leq\frac{C}{|x|^\zeta}\int_\R|\Phi_t(x-y-z)|\, |f(z)|\,|z|^\zeta\,dz\\
&\leq\frac{C}{|x|^\zeta}\,M(|f|\,|\cdot|^\zeta)(x),
\end{split}
\end{equation}
due to \eqref{C213}, where $\zeta\in(0,\beta)$ is determined later.
We use similar way to control $\mathbf{I}_3$ by
\begin{equation}\label{C211}
\begin{split}
|\mathbf{I}_3|&=\left|\int_\R\Phi_t(x-y)\, f(z)\left(1-\chi_j(z)\right)\,dz\right|\\
&\leq\frac{1}{|x|^2}\int_\R\left(\frac{|x|^2}{t^2}\,\left|\Phi\big(\frac{x}{t}-\frac{y}{t}\big)\right|\right) |f(z)|\left(1-\chi_j(z)\right)\,dz\\
&\leq\frac{C}{|x|^{2+\beta}}\int_\R|f(z)|\,|z|^\beta\,dz.
\end{split}
\end{equation}

Substituting \eqref{C210}--\eqref{C211} into \eqref{C208} and taking the supremum over $t>0$, $|y|\le t$, and all $\Phi\in\mathcal{S}_\mathcal{F}$, we obtain that for any $x$ with $|x|\geq 1$,
\begin{equation}\label{C216}
\mathcal{M}f(x)\,(1+|x|)^\alpha\leq \frac{C}{|x|^{2+\beta-\alpha}}\int_\R|f(z)|\,|z|^\beta\,dz+\frac{C}{|x|^{\zeta-\alpha}}M(|f|\,|\cdot|^\zeta)(x).
\end{equation}
Note that $(p\beta+2(p-1))/(\beta+(2+\alpha)(p-1))\in(1,p)$, then we take 
$$q=\frac{1}{2}\left(1+\frac{p\beta+2(p-1)}{\beta+(2+\alpha)(p-1)}\right)\in(1,p),
\quad
\zeta=\frac{(p-q)\beta}{(p-1)q}\in\left(\alpha+2-\frac{2}{q},\frac{\beta}{q}\right).$$
Thereby from \eqref{C216}, we derive that for $q'=q/(q-1)$,
\begin{equation}\label{C217}
\begin{split}
&\int_{\{|x|\geq 1\}}\mathcal{M}f(x)\,(1+|x|)^\alpha\,dx\\
&\leq\int_{\{|x|\geq 1\}}\frac{C}{|x|^{2+\beta-\alpha}}\,dx\int_\R|f(z)|\,|z|^\beta\,dz+\int_{\{|x|\geq 1\}}\frac{C}{|x|^{\zeta-\alpha}}\,M(|f|\,|\cdot|^\zeta)(x)\,dx\\
&\leq C\int_\R|f(x)|\,|x|^\beta\,dx+\left(\int_{\{|x|\geq 1\}}\frac{C}{|x|^{q'(\zeta-\alpha)}}\right)^{1/q'}\left(\int_\R M^q(|f|\,|\cdot|^\zeta)(x)\,dx\right)^{1/q}\\
&\leq C \int_\R|f(x)|\,|x|^\beta\,dx+C\left(\int_\R\left(|f(x)|\,|x|^\zeta\right)^q\,dx\right)^{1/q}\\
&\leq C\int_\R|f(x)|\,|x|^\beta\,dx+C\,\|f\|_{L^p},
\end{split}
\end{equation}
where we have used the fact that $M$ maps $L^q(\R)$ to itself (see \cite[Page 13]{S}) and
\begin{equation}\label{C214}\notag
\left(\int_\R\left(|f(x)|\,|x|^\zeta\right)^q\,dx\right)^{1/q}\leq \|f\|_{L^p}^{1-\zeta/\beta}\left(\int_\R|f(x)|\,|x|^\beta\,dx\right)^{\zeta/\beta}.
\end{equation}

Combining \eqref{C207} with \eqref{C217} gives \eqref{C203} and finishes the proof of
Lemma \ref{L22}.
\end{proof}

\section{Hardy spaces and whole-space velocity control}
\label{S3}
\quad This section is devoted to proving Theorem \ref{T3}, which identifies the far-field velocity $u_\infty$ and controls the integrable remainder $v=u-u_\infty$ throughout $\R$, including the vacuum region. 
In view of the Caffarelli-Kohn-Nirenberg inequality,
$$u=u_\infty+v,~~\|v\|_{L^{2/\beta}}\leq C\|\nabla u\cdot|x|^\beta\|_{L^2},~~0<\beta<1,$$
the key step is to establish a weighted $L^2$ estimate of $\nabla u$, 
which relies on the cancellation structure of \eqref{C101}$_2$.

To illustrate this, we consider the Poisson system, which retains the principal elliptic structure of the momentum equation:
\begin{equation}\label{C301}
-\Delta u=f\quad\mathrm{in}~\mathbb{R}^2.
\end{equation}
Here $f$ models the inertial term $\rho\dot u$; for the moment, we assume $f\in C_0^\infty(\R)$. Let us introduce a truncation function $\varphi\in C_0^\infty(\R)$ defined by
\begin{equation}\label{var}
\varphi(x)=
\begin{cases}
1,\quad |x|\leq 1,\\
0,\quad |x|\geq 2,\end{cases}
 0\leq\varphi(x)\leq 1,\quad |\nabla\varphi(x)|\leq 2,\quad\text{for any}\, x\in\R.
\end{equation}
Multiplying \eqref{C301} by $\varphi_R(x)\triangleq\varphi(x/R)$ and integrating by parts yield that for any $\alpha\geq 0$,
\begin{equation*}
\begin{split}
\left|\int_\R f\cdot\varphi_R\,dx\right|
&=\left|\int_\R\nabla u\cdot\nabla\varphi_R\,dx\right|
\leq 2\left(\frac{3\pi}{R^{2\alpha}}\int_{\{|x|\geq R\}}|\nabla u|^2\,|x|^{2\alpha}\,dx\right)^{1/2}.
\end{split}
\end{equation*}
If $|\nabla u|\,|x|^\alpha\in L^2(\R)$, letting $R\rightarrow\infty$, we obtain the necessary cancellation condition,  
\begin{equation}
\label{C302}
\int_\R f\,dx=\lim_{R\rightarrow\infty}\int_\R f\cdot\varphi_R\,dx
=0,
\end{equation}
which is not true for general $f\in C_0^\infty(\R)$.

Conversely, for the Newtonian solution of \eqref{C301} with $f\in C_0^\infty(\R)$, the cancellation \eqref{C302} is sufficient to obtain $|\nabla u|\,|x|^\alpha\in L^2$ for $0<\alpha<1$ by
\begin{equation}\label{C345}
\left(\int_\R|\nabla u|^2\,|x|^{2\alpha}\,dx\right)^{1/2}\leq C\int_\R\mathcal{M}f\,|x|^{\alpha}\,dx,
\end{equation}
where $\mathcal{M}$ is the grand maximal operator in Definition \ref{D21} and the right hand side is finite due to Lemma \ref{L22}.
We remark that \eqref{C345} is inspired by its unweighted counterpart:
when $f\in C_0^\infty(\R)$ has zero-mean, Lemma \ref{L22} gives $f\in\mathcal{H}^1(\R)$, and \eqref{C202} along with Sobolev embedding provides
\begin{equation}\label{C303} 
\|\nabla u\|_{L^2}\leq C\|\nabla^2 u\|_{L^1}= 
C\sum_{i,j}\|R_iR_j f\|_{L^1}\leq C\|f\|_{\mathcal{H}^1(\R)}=C\|\mathcal{M}f\|_{L^1},
\end{equation}
where $R_i$ $(i=1,2)$ is the $i$-th Riesz transform in $\R$.

We first establish the weighted estimates \eqref{C345} and apply them to the Stokes system. 
After verifying the cancellation conditions, we complete the proof of Theorem 1.1.

\subsection{Endpoint estimates for weighted Riesz potentials}\label{S31}
\quad Recall the Riesz potential $I_s(f)$ defined for $f\in\mathcal{S}$ and $0<s<2$,
$$I_s(f)(x)\triangleq\int_\R\frac{1}{|x-y|^{2-s}}f(y)\,dy.$$
The weighted Hardy-Littlewood-Sobolev inequality, initially introduced by Stein-Weiss \cite{SGW} (see also \cite{SW,NP}), ensures that for any  $\alpha$,  $\beta$, and $1< p\leq q<\infty$ satisfying 

\begin{equation*} 
\begin{split}
&q\alpha>-2\,\big(\mathrm{integrable}~\mathrm{condition}\big),\quad
\frac{1}{q}=\frac{1}{p}+\frac{\beta-\alpha-s}{2}\,\big(\mathrm{due}~\mathrm{to}~\mathrm{scaling}\big),\\
&\alpha\leq\beta\, \big(\mathrm{due}~\mathrm{to}~\mathrm{translation}\big),\quad\quad\quad\beta<\frac{2(p-1)}{p}\,\big(\mathrm{Muckenhoupt}~ \mathrm{condition}\big),
\end{split}
\end{equation*}
there is a generic positive constant $C$ decided by $p$, $s$, $\alpha$, and $\beta$ such that
\begin{equation}
\label{C307}
\left(\int_\R|I_sf(x)|^q\,|x|^{q\alpha}\,dx\right)^{1/q}\leq C\left(\int_\R|f(x)|^p\,|x|^{p\beta}\,dx\right)^{1/p}.
\end{equation}

At $p=1$, the corresponding strong-type estimates can hold only if $\alpha<\beta<0$, see \cite{SW}.
Motivated by \eqref{C303}, we apply cancellation to establish the following Hardy-type estimates for Riesz potentials. It allows $0<\beta<1$
and includes the critical case $\alpha=\beta$, while preserving the same scaling relation.
\begin{lemma}\label{L31}
Suppose that $f\in C_0^\infty(\R)$ satisfies $\int_\R f\,dx=0$, and the real numbers $0<s<2$, $q\geq 1$, $\alpha$, and $\beta$ are subject to 
\begin{equation}\label{C308}
\frac{1}{q}=1+\frac{\beta-\alpha-s}{2},\quad \alpha\leq\beta,\quad 0<\beta<1,
\end{equation}
which automatically ensures $q\alpha>-2$. Then there is a positive constant $C$ determined by $s$, $\alpha$, and $\beta$ such that
\begin{equation}\label{C309}
\left(\int_\R|I_sf(x)|^q\,|x|^{q\alpha}\,dx\right)^{1/q}\leq C\int_\R \mathcal{M}f(x) \, |x|^{\beta}\,dx,
\end{equation}
where $\mathcal{M}$ is the grand maximal operator introduced in Definition \ref{D21}.
\end{lemma}
\begin{remark}
The first two restrictions  in \eqref{C308} are due to scaling and translation, which can not be improved. While $0<\beta<1$ seems also optimal under the cancellation condition $\int_\R f(x)\,dx=0$, which can be extended to $\beta\geq 1$ by imposing a stronger cancellation condition on $f$ (see Remark \ref{re2.1}).
\end{remark}
\begin{proof}
By the cancellation condition $\int_\R f\,dx=0$, we decompose $I_sf(x)$ into three parts. Precisely, for any $x\neq {0}$, there is some $j\in\mathbb{Z}$ such that $2^{j}\leq |x|<2^{j+1}$, we write
\begin{equation}\label{C310}
\begin{split}
I_sf(x)
&=\int_\R\left(\frac{1}{|x-y|^{2-s}}-\frac{1}{|x|^{2-s}}\right)f(y)\,dy\\
&=\int_\R\left(\frac{1}{|x-y|^{2-s}}-\frac{1}{|x|^{2-s}}\right)f(y)\chi_j(y)\,dy\\
&\quad+\int_\R \frac{1}{|x-y|^{2-s}} f(y)(1-\chi_j(y))\,dy-\int_\R\frac{1}{|x|^{2-s}}f(y)(1-\chi_j(y))\,dy\\
&\triangleq\mathbf{J}_1(x)+\mathbf{J}_2(x)+\mathbf{J}_3(x),
\end{split}
\end{equation}
where $\chi_j(x)=\chi(x/2^{j})$ and $\chi(x)\in C_0^\infty(\R)$ is a truncation function as in \eqref{chi}.

For $\mathbf{J}_1(x)$, because $|y|\leq |x|/2$ in the support of $\chi_j(y)$, we check that
\begin{equation}\label{C312}
\begin{split}
|\mathbf{J}_1(x)|&=\left|\int_\R\left(\frac{1}{|x-y|^{2-s}}-\frac{1}{|x|^{2-s}}\right)f(y)\chi_j(y)\,dy\right|\\
&\leq C\int_\R\left(\int_0^1\frac{|y|}{|x-\tau y|^{3-s}}\,d\tau\right)|f(y)|\,\chi_j(y)\,dy\\
&\leq \frac{C}{|x|^{3-s}}\int_\R|y|\,|f(y)|\,\chi_j(y)\,dy\\
&\leq\frac{C}{|x|^{3-s}}\int_{\{|y|\leq|x|\}} |f(y)| \,|y|\,dy.
\end{split}
\end{equation}
Note that $0<\beta<1$ implies
\begin{equation}\label{y1}q(3-\alpha-s)>2,\quad q(2-\alpha-s)<2,\end{equation}
then integrating \eqref{C312} with respect to $x$ and using  integration by parts in polar coordinates lead to 
\begin{equation*}
\begin{split}
&\int_\R|\mathbf{J}_1(x)|^q\,|x|^{q\alpha}\,dx\\
&\leq\int_\R\frac{C}{|x|^{q(3-\alpha-s)}}\left(\int_{\{|y|\leq|x|\}} |f(y)|\,|y|\,dy\right)^q\,dx\\
&=\int_0^\infty\frac{C}{r^{q(3-\alpha-s)-1}}\left(\int_{\{|y|\leq r\}} |f(y)|\,|y|\,dy\right)^q\,dr\\
&=\int_0^\infty \frac{C}{r^{q(3-\alpha-s)-2}}\left(\int_{\{|y|\leq r\}} |f(y)|\,|y|\,dy\right)^{q-1}\left(\int_{\{|y|=r\}}|f(y)|\,dS\right)r\,dr\\
&=C\int_\R\left(\int_{\{|y|\leq|x|\}} |f(y)|\,|y|\,dy\right)^{q-1}{|f(x)|}\,{|x|^{3-q(3-\alpha-s)}}\,dx.
\end{split}
\end{equation*}
The direct calculation shows $(q-1)(1-\beta)+3-q(3-\alpha-s)=\beta$, thus we have
\begin{equation}\label{C322}
\begin{split}
\int_\R|\mathbf{J}_1(x)|^q\,|x|^{q\alpha}\,dx
&\leq C\int_\R\left(\int_{\{|y|\leq|x|\}} |f(y)|\,|y|\,dy\right)^{q-1}{|f(x)|}\,{|x|^{3-q(3-\alpha-s)}}\,dx\\
&\leq C\int_\R\left(\int_{\{|y|\leq|x|\}} |f(y)|\,|y|^\beta\,dy\right)^{q-1}|f(x)|\,|x|^{\beta}\,dx\\ 
&\leq C\left(\int_\R \mathcal{M}f(x)\,|x|^\beta\,dx\right)^q,
\end{split}
\end{equation}
where the last line is due to $|f(x)|\leq\mathcal{M}f(x)$.

Similarly, since $1-\chi_j(y)$ is  supported in $\{y:|y|\geq 2^{j-2}\}$, we obtain
\begin{equation}\label{C320}
\begin{split}
|\mathbf{J}_3(x)|=\frac{1}{|x|^{2-s}}\left|\int_\R f(y)\left(1-\chi_j(y)\right)\,dy\right| \leq \frac{C}{|x|^{2-s}}\int_{\{|y|\geq |x|/8\}} |f(y)|\,dy.
\end{split}
\end{equation}
Then in view of \eqref{y1}, we use  integration by parts in polar coordinates to find
\begin{equation*}
\begin{split}
&\int_\R|\mathbf{J}_3(x)|^q\,|x|^{q\alpha}\,dx\\
&\leq\int_\R\frac{C}{|x|^{q(2-\alpha-s)}}\left(\int_{\{|y|\geq |x|/8\}} |f(y)|\,dy\right)^q\,dx\\
&=\int_0^\infty \frac{C}{r^{q(2-\alpha-s)-2}}\left(\int_{\{|y|\geq r/8\}} |f(y)|\,dy\right)^{q-1}\left(\int_{\{|y|=r\}}|f(y)|\,dS\right)\,dr\\
&=C\int_\R\left(\int_{\{|y|\geq |x|/8\}} |f(y)|\,dy\right)^{q-1}{|f(x)|}\,{|x|^{2-q(2-\alpha-s)}}\,dx.\\
\end{split}
\end{equation*}
Moreover,  $-\beta(q-1)+2-q(2-\alpha-s)=\beta$ ensures that
\begin{equation}\label{C323}
\begin{split}
\int_\R|\mathbf{J}_3(x)|^q\,|x|^{q\alpha}\,dx
&\leq C\int_\R\left(\int_{\{|y|\geq |x|/8\}} |f(y)|\,dy\right)^{q-1}{|f(x)|}\,{|x|^{2-q(2-\alpha-s)}}\,dx\\
&\leq C\int_\R\left(\int_{\{|y|\geq |x|/8\}} |f(y)|\,|y|^\beta\,dy\right)^{q-1}|f(x)|\,|x|^{\beta}\,dx\\
&\leq C\left(\int_\R \mathcal{M}f(x)\,|x|^\beta\,dx\right)^q.
\end{split}
\end{equation}

Next, we estimate $\mathbf{J}_2(x)$ based on a scaling argument. 
Let us introduce the dyadic partition of $\R$ by defining  that for any $i\in\mathbb{Z}$,
\begin{equation*}
\begin{split}
\phi_i(x)\triangleq\chi_i(x)-\chi_{i-1}(x)&,\quad \psi_i(x)
\triangleq \phi_i(x)/\sum_{k=-\infty}^\infty\phi_k(x),\\
f_i(x)\triangleq f(2^ix)&,\quad x_i\triangleq x/2^i.
\end{split}
\end{equation*}
Note that $\psi_i$ is well defined and $\sum_i\psi_i=1$ in $\R\setminus \{0\}$, since each $x\in\R\setminus\{0\}$ lies in at most three of $\{\mathrm{supp}\phi_k\}_{k\in\mathbb{Z}}$.
Then it holds that
\begin{equation}\label{C313}
\begin{split}
\mathbf{J}_2(x)&=\int_\R \frac{1}{|x-y|^{2-s}}\, f(y)\left(1-\chi_j(y)\right)\,dy\\
&=2^{js}\int_\R\frac{1}{|x_j-y|^{2-s}}\, f_j(y)\left(1-\chi(y)\right)\,dy\\
&=2^{js}\sum_{i=-\infty}^m \int_{\R}\frac{\psi_i(x_j-y)}{|x_j-y|^{2-s}}\, f_j(y)\left(1-\chi(y)\right)\,dy\\
&\quad +2^{js}\sum_{i=m+1}^{\infty}\int_{\R}\frac{\psi_i(x_j-y)}{|x_j-y|^{2-s}}\, f_j(y)\left(1-\chi(y)\right)\,dy\\
&\triangleq 2^{js}\big(\mathbf{J}_{21}(x)+\mathbf{J}_{22}(x)\big),
\end{split}
\end{equation}
for some $m$ determined later. 

We first handle $\mathbf{J}_{21}(x)$. 
Denote the family of $C_0^\infty$ kernels by
$$\Theta_i(x_j,y)\triangleq\frac{\psi_i(x_j-y)}{|x_j-y|^{2-s}}\left(1-\chi(y)\right),\quad \mathrm{where}~|x_j|< 2,$$
we can rewrite $\mathbf{J}_{21}(x)$ as
\begin{equation}\label{C314}
\begin{split}
\mathbf{J}_{21}(x)=\sum_{i=-\infty}^m \int_{\R}\Theta_i(x_j,y)\, f_j(y)\,dy.
\end{split}
\end{equation}
The following crucial step is to control each term  in \eqref{C314} by $\mathcal{M}f_j(x_j)$.

When $i\geq 5$, $\psi_i(x_j-y)$ is supported on the set $\{y:|x_j-y|\geq 4\}\subset\{y:\,|y|\geq 2\}$, in which $1-\chi(y)\equiv 1$, therefore it holds that
\begin{equation*}\begin{split}
&\Theta_i(x_j,y)=\frac{\psi_i(x_j-y)}{|x_j-y|^{2-s}}=2^{is}\,\Psi_{2^i}(x_j-y)\\
&\mathrm{with}\quad \Psi(y)
=\frac{\psi_0(y)}{|y|^{2-s}},\quad \Psi_{2^{i}}(y)=2^{-2i}\,\Psi\left(\frac{y}{2^i}\right).
\end{split}
\end{equation*}
Then for semi-norms $\|\cdot\|_{a,b}$ given by Definition \ref{D21} with $|a|,|b|\leq 3$, we can find a positive constant $C$ determined by $\chi$ such that
$\|\Psi\|_{a,b}\leq C$. Thus, $\Psi/C\in\mathcal{S}_\mathcal{F}$ as in \eqref{C206} and the definition of $\mathcal{M}$ in \eqref{C204} implies that
\begin{equation}\label{C315}
\begin{split}
\left|\int_\R\Theta_i(x_j,y)\, f_j(y)\,dy\right|
=&2^{is}\left|
\int_\R\Psi_{2^{i}}(x_j-y)\, f_j(y)\,dy\right| 
\leq  C\,2^{is}\,\mathcal{M}f_j(x_j).
\end{split}
\end{equation}
When $i<5$, we calculate that
\begin{equation*}
\begin{split}
&\Theta_i(x_j,y)=\frac{\psi_i(x_j-y)}{|x_j-y|^{2-s}}\left(1-\chi(y)\right)=2^{is}\, \Psi^i_{2^i}(x_j-y),\\
&\mathrm{with}
\quad {\Psi}^i(y)=\frac{\psi_0(y)}{|y|^{2-s}}\left(1-\chi(x_j-2^iy)\right),\quad \Psi^i_{2^{i}}(y)=2^{-2i}\, \Psi^i\left(\frac{y}{2^i}\right).
\end{split}
\end{equation*}
For semi-norms $\|\cdot\|_{a,b}$ with $|a|,|b|\leq 3$, we carefully check that  there is a positive constant $C$ independent of $i$ and $x_j$ such that
$\|\Psi^i\|_{a,b}\leq C$. Hence, $\Psi^i/C\in\mathcal{S}_\mathcal{F}$ for each $i$  and the definition of $\mathcal{M}$ in \eqref{C204} implies that
\begin{equation}\label{C316}
\begin{split}
\left|\int_{\R}\Theta_i(x_j,y)\, f_j(y)\,dy\right|=2^{is}\left|
\int_\R \Psi^i_{2^{i}}(x_j-y)\, f_j(y)\,dy\right|\leq C\,2^{is}\,\mathcal{M}f_j(x_j).
\end{split}
\end{equation}
Substituting \eqref{C315} and \eqref{C316} into \eqref{C314} provides that
\begin{equation}\label{C317}
|\mathbf{J}_{21}(x)|\leq C\sum_{i=-\infty}^m2^{is}\,\mathcal{M}f_j(x_j)= C\,2^{ms}\,\mathcal{M}f_j(x_j)=C\,2^{ms}\,\mathcal{M}f(x),
\end{equation}
where the last equality is due to the scaling property of $\mathcal{M}$. 

The estimate on $\mathbf{J}_{22}(x)$ is simpler. In fact, from the definition of $\psi_i$, we get
\begin{equation}
\begin{split}\label{C317+}
|\mathbf{J}_{22}(x)|&\leq\sum_{i=m+1}^\infty \left|\int_{\R}\frac{\psi_i(x_j-y)}{|x_j-y|^{2-s}} \left(1-\chi(y)\right) f_j(y)\,dy\right|\\
&\leq\sum_{i=m+1}^\infty\frac{C}{2^{i(2-s)}}\int_\R\left(1-\chi(y)\right)\, |f_j(y)|\,dy\\
&\leq\sum_{i=m+1}^\infty\frac{C}{2^{i(2-s)}}\int_{\{|y|\geq 1/4\}} |f_j(y)|\,dy\\
&\leq\sum_{i=m+1}^\infty\frac{C}{2^{i(2-s)+2j}}\int_{\{|y|\geq|x|/8\}}|f(y)|\,dy \\
&=\frac{C}{2^{m(2-s)+2j}}\int_{\{|y|\geq|x|/8\}}|f(y)|\,dy,
\end{split}
\end{equation}
where the last second line is due to scaling and $2^{j}\leq|x|<2^{j+1}$ in current settings. 

Now, we select $m$ such that
\begin{equation*}
2^{ms}\,\mathcal{M}f(x)\sim \frac{1}{2^{m(2-s)+2j}}\int_{\{|y|\geq|x|/8\}}|f(y)|\,dy,
\end{equation*}
then we deduce from  \eqref{C313}, \eqref{C317}, and \eqref{C317+} that
\begin{equation}\label{C319}
\begin{split}
|\mathbf{J}_2(x)|&\leq 2^{js}\big(|\mathbf{J}_{21}(x)|+|\mathbf{J}_{22}(x)|\big)\\
&\leq C2^{js}\left(2^{ms}\,\mathcal{M}f(x)+\frac{1}{2^{m(2-s)+2j}}\, \int_{\{|y|\geq|x|/8\}}|f(y)|\,dy \right)\\
&\le C\,(\mathcal{M}f)^{1-s/2}(x)\,\left(\int_{\{|y|\geq|x|/8\}}|f(y)|\,dy\right)^{s/2}.\\
\end{split}
\end{equation}
Note that $q(2-s)/2\leq 1$ due to $\beta\geq\alpha$. When $q=2/(2-s)$ and $\beta=\alpha$, we have
\begin{equation}\label{C327}
\begin{split}
\int_\R|\mathbf{J}_2(x)|^q\,|x|^{q\alpha}\,dx 
&\leq
C\int_\R\mathcal{M}f(x)\,|x|^{q\beta}\,\left(\int_{\{|y|\geq|x|/8\}}|f(y)|\,dy\right)^{q-1}\,dx\\
&\leq C\int_\R\mathcal{M}f(x) \,|x|^\beta\,\left(\int_{\{|y|\geq|x|/8\}}|f(y)|\,|y|^\beta\,dy\right)^{q-1}dx\\
&\leq C\left(\int_\R \mathcal{M}f(x)\,|x|^\beta\,dx\right)^q,
\end{split}
\end{equation}
where the last line is due to $|f(x)|\leq\mathcal{M}f(x)$. While for $q(2-s)/2< 1$ and $\beta>\alpha$, the H\"{o}lder inequality  yields
\begin{equation}\label{C324}
\begin{split}
&\int_\R|\mathbf{J}_2(x)|^q\,|x|^{q\alpha}\,dx\\
&\leq
C\int_\R(\mathcal{M}f)^{\frac{q(2-s)}{2}}(x)\,|x|^{q\alpha}\,\left(\int_{\{|y|\geq|x|/8\}}|f(y)|\,dy\right)^{\frac{qs}{2}}\,dx\\
&\leq C
\left(\int_\R|x|^{\frac{(2\alpha-(2-s)\beta)q}{2-q(2-s)}}\left(\int_{\{|y|\geq|x|/8\}}|f(y)|\,dy\right)^{\frac{qs}{2-q(2-s)}}\,dx\right)^\frac{2-q(2-s)}{2}\\
&\qquad\cdot\left(\int_\R\mathcal{M}f(x)\,|x|^{\beta}\,dx\right)^{\frac{q(2-s)}{2}}.
\end{split}
\end{equation}
We check that $\big(2\alpha-(2-s)\beta\big)q/{\big(2-q(2-s)\big)}>-2$ due to $\beta\cdot s>0$, hence the integration by parts in polar coordinates is valid, that is,
\begin{equation}
\begin{split}\label{C325}
&\int_\R|x|^{\frac{(2\alpha-(2-s)\beta)q}{2-q(2-s)}}\left(\int_{\{|y|\geq|x|/8\}}|f(y)|\,dy\right)^{\frac{qs}{2-q(2-s)}}dx\\
&=C\int_0^\infty r^{\frac{(2\alpha-(2-s)\beta)q}{2-q(2-s)}+2}\left(\int_{\{|y|\geq r/8\}} |f(y)|\,dy\right)^{\frac{2(q-1)}{2-q(2-s)}}\left(\int_{\{|y|=r/8\}}|f(y)|\,dS\right)\,dr\\
&=C\int_\R\left(\int_{\{|y|\geq |x|/8\}} |f(y)|\,dy\right)^{\frac{2(q-1)}{2-q(2-s)}}|f(x)|\,|x|^{\frac{(2\alpha-(2-s)\beta)q}{2-q(2-s)}+2}\,dx\\
&\leq C\int_\R\left(\int_{\{|y|\geq |x|/8\}} |f(y)|\,|y|^\beta\,dy\right)^{\frac{2(q-1)}{2-q(2-s)}}|f(x)|\,|x|^{\beta}\,dx\\
&\leq C\left(\int_\R \mathcal{M}f(x)\,|x|^\beta\,dx\right)^{\frac{qs}{2-q(2-s)}},
\end{split}
\end{equation}
where we have used  $|f(x)|\leq\mathcal{M}f(x)$ and 
$$-\frac{2(q-1)\beta}{2-q(2-s)}+\frac{(2\alpha-(2-s)\beta)q}{2-q(2-s)}+2=\beta.$$
Putting \eqref{C325} into \eqref{C324} and combining \eqref{C327}, we arrive at 
\begin{equation}\label{C326}
\int_\R|\mathbf{J}_2(x)|^q\,|x|^{q\alpha}\,dx\leq C\left(\int_\R\mathcal{M}f(x)\,|x|^{\beta}\,dx\right)^{q}.
\end{equation}

Collecting \eqref{C310}, \eqref{C322}, \eqref{C323}, and \eqref{C326}, we conclude that
\begin{equation*}
\begin{split}
\int_\R|I_sf(x)|^q\,|x|^{q\alpha}\,dx\le C\sum_{i=1}^3\int_\R|\mathbf{J}_i(x)|^q\,|x|^{q\alpha}\,dx\leq C\left(\int_\R \mathcal{M}f(x)\,|x|^\beta\,dx\right)^q,
\end{split}
\end{equation*}
which finishes the proof of Lemma \ref{L31}.
\end{proof}

\subsection{Weighted $L^2$ estimates for the Stokes system}\label{S32}
\quad We rewrite the momentum equation \eqref{C101}$_2$ as the following Stokes system,
\begin{equation}\label{C331}
\begin{cases}
-\Delta u+\nabla P=\mathrm{div}\big(2(\mu(\rho)-1)Du\big)-\rho\dot u,\\
\mathrm{div}u=0.
\end{cases}
\end{equation}
Under the cancellation condition $\int_\R\rho\dot u\,dx=0$, Lemma \ref{L31} provides a priori weighted estimates of $\nabla u$.
\begin{lemma}\label{L32}
Let $(\rho, u, P)$ satisfy \eqref{C331} and
$$0\leq\rho\in L^\infty(\R),~~\rho\,(1+|x|^2)\in L^1(\R),~~\nabla u\in L^2(\R),~~\sqrt{\rho}\dot{u}\in L^2(\R).$$
Assume that $\rho\dot u$ satisfies the cancellation condition
$
\int_\R\rho\dot{u}\,dx=0.$
For any $0<\alpha<1$, there are constants $\varepsilon_0$ and $C_{\alpha}$ determined by  $\alpha$, such that, if
\begin{equation*} 
\|\mu(\rho)-1\|_{L^\infty}(t)\leq\varepsilon_0.
\end{equation*}
Then we have
\begin{equation}\label{C332}
\left(\int_\R|\nabla u(x,t)|^2\,|x|^{2\alpha}\,dx\right)^{1/2}\leq C_{\alpha}\int_\R\mathcal{M}(\rho\dot u)(x,t)\,|x|^\alpha\,dx.
\end{equation} 
Moreover, for every $0<\alpha<\beta<1$, there exists a constant $C_{\alpha,\beta}$ depending on $\alpha$ and $\beta$,
such that
\begin{equation}\label{C339}
\begin{split}
&\left(\int_\R|\nabla u(x,t)|^2\,|x|^{2\alpha}\,dx\right)^{1/2}+\int_\R\mathcal{M}(\rho\dot u)(x,t)\,(1+|x|^\alpha)\,dx\\
&\leq C_{\alpha,\beta}\left(\|\rho\|_{L^1}^{1/2}+\|\rho\|_{L^\infty}^{1/2}+\|\rho\|_{L^1}^{(1-\beta)/2}\|\rho\,|x|^2\|_{L^1}^{\beta/2}\right)\,\|\sqrt{\rho}\dot u\|_{L^2}.
\end{split}
\end{equation}
Here $\mathcal{M}$ is the grand maximal operator introduced in Definition \ref{D21}.
\end{lemma}

\begin{proof}
We give the a priori calculation below; its justification follows from a standard perturbation argument in the intersection of the weighted and unweighted \(L^2\) spaces.

According to \eqref{C331} and \eqref{C334} in Lemma \ref{pl3}, we obtain that for the operator ``$\mathbb{P}=\mathrm{Id}-\nabla\Delta^{-1}\mathrm{div}$",
\begin{equation}
\begin{split}\label{du}
\nabla u&=
-\nabla\Delta^{-1}\mathbb{P}\,\mathrm{div}\left(2(\mu(\rho)-1)Du\right)+\nabla\Delta^{-1}\mathbb{P}(\rho\dot u)\\
&=\mathbf{T}_1\left(2(\mu(\rho)-1)Du\right)-\mathbf{T}_2(\rho\dot u)\\
&=2\int_\R {K}_1(x-y)\cdot (\mu(\rho)-1)Du(y)\,dy-\int_\R {K}_2(x-y)\cdot\rho\dot u(y)\,dy.\\
\end{split}
\end{equation}
Observe that the kernel ${K}_1(x)$ satisfies that
\begin{equation*}
\begin{split} 
&\big|{K}_1(x)\big|\leq\frac{C}{|x|^2}, \quad \mathrm{for}~\mathrm{all}~ x\in\R\setminus \{0\},\\
&\big|{K}_1(x-y_1)-{K}_1(x-y_2)\big|\leq C\,\frac{|y_1-y_2|}{|x-y_2|^3},\quad \mathrm{if}~|y_1-y_2|\leq |x-y_2|/2,\\
&\left\|K_1\ast f\right\|_{L^2}\leq C\,\|f\|_{L^2},\quad \mathrm{for}~\mathrm{all}~f\in C_0^\infty(\R),
\end{split}
\end{equation*}
where $C$ is determined by the dimension due to \cite[Chapter I]{S}, and the kernel  ${K}_2(x)$ satisfies that
\begin{equation}\label{k2}
|\nabla^m K_2(x)|\le C(m)|x|^{-(m+1)},\quad \mathrm{for}~\mathrm{all}~ x\in\R\setminus \{0\},\, m\in \mathbb{N}.
\end{equation}

 Therefore, according to \cite[Chapter V]{S}, for $A_2$-weight $|x|^{2\alpha}$ ($0<\alpha<1$), there is a positive constant $C$ decided by $\alpha$ such that for any $f$ satisfying $f|x|^\alpha\in L^2$,
\begin{equation}\label{C346}\int_\R|\mathbf{T}_1f|^2\,|x|^{2\alpha}\,dx\leq C\int_\R|f|^2\,|x|^{2\alpha}\,dx.
\end{equation}
Set $f=2(\mu(\rho)-1)Du$ in the above estimate, then we admit
\begin{equation}\label{C335}
\begin{split}
\int_\R\left|\mathbf{T}_1\left(2(\mu(\rho)-1)Du\right)\right|^2\,|x|^{2\alpha}\,dx&\leq C\int_\R(\mu(\rho)-1)^2|Du|^2\,|x|^{2\alpha}\,dx\\
&\leq\frac{1}{4}\int_\R|\nabla u|^2\,|x|^{2\alpha}\,dx,
\end{split}
\end{equation}
by setting $\|\mu(\rho)-1\|_{L^\infty}\leq\varepsilon_0\leq (4C)^{-1/2}$ with $C$ given by \eqref{C346}.

Since the kernel of $\mathbf{T}_2$ satisfies the same size and derivative bounds as that of $I_1$, we can apply \eqref{C309} to find a constant $C$ determined by $\alpha$ such that
\begin{equation}\label{C336}
\begin{split}
\left(\int_\R|\mathbf{T}_2(\rho\dot u)|^2\,|x|^{2\alpha}\,dx\right)^{1/2}\leq C\int_\R\mathcal{M}(\rho\dot u)\,|x|^{\alpha}\,dx.
\end{split}
\end{equation}

Combining \eqref{du} with \eqref{C335} and  \eqref{C336}, we obtain  that if $\|\mu(\rho)-1\|_{L^\infty}\le\varepsilon_0$,
\begin{equation*}
\left(\int_\R|\nabla u|^2\,|x|^{2\alpha}\,dx\right)^{1/2} \leq C\int_\R\mathcal{M}(\rho\dot u)\,|x|^\alpha\,dx,
\end{equation*}
which gives \eqref{C332}. Moreover,
in view of $\int_\R\rho\dot u(x)\,dx=0$ and \eqref{C203}, we have
\begin{equation}\label{C340}
\int_\R\mathcal{M}(\rho\dot u)\,(1+|x|)^\alpha\,dx\leq C\left(\int_\R \rho|\dot{u}|\,(1+|x|)^\beta\,dx+\|\rho\dot{u}\|_{L^2}\right).
\end{equation}
Let us check that
\begin{equation}\label{C341}
\|\rho\dot u\|_{L^1}+\|\rho\dot u\|_{L^2}\leq \left(\|\rho\|_{L^1}^{1/2}+\|\rho\|_{L^\infty}^{1/2}\right)\|\sqrt{\rho}\dot u\|_{L^2}, 
\end{equation}
and by the interpolation,
\begin{equation}\label{C342}
\begin{split}
\int_\R \rho|\dot u|\,|x|^\beta\,dx&\leq\left(\int_\R\rho\, |x|^{2\beta}\,dx\right)^{1/2}\left(\int_\R\rho|\dot u|^2\,dx\right)^{1/2}\\
&\leq\|\rho\|_{L^1}^{(1-\beta)/2}\,\|\rho\,|x|^2\|_{L^1}^{\beta/2}\,\|\sqrt{\rho}\dot u\|_{L^2}.
\end{split}
\end{equation}

Substituting \eqref{C341} and \eqref{C342} into \eqref{C340} and using \eqref{C332}, we thus obtain \eqref{C339} and
the proof of Lemma \ref{L32} is therefore finished.
\end{proof}
\begin{remark}\label{R32}
We mention that when $0<\alpha<\beta<1/2$, the constant $\varepsilon_0$ in Lemma \ref{L32} can be selected independent of $\alpha$. Since in such case, the index $2\alpha(<1)$ in \eqref{C346} is strictly away from the critical value $2$ for the $A_2$-weight estimates in $\R$ $($see \cite[Page 218]{S}$),$ thus the constant $C$ in \eqref{C346} can be chosen independent of $\alpha$.
\end{remark}

\subsection{Cancellation and velocity reconstruction}
\quad
By virtue of Lemma \ref{L32}, the trigger of $L^2$ weighted estimates of $\nabla u$ at any time is exactly the cancellation condition:
\begin{equation}\label{C401}
\int_\R\rho\dot u(x,t)\,dx=0.
\end{equation}
We establish it by using compactly supported divergence-free test fields that approximate constant vectors, see \cite{SV}, moreover the same cutoff construction also gives conservation of the total momentum.
These two cancellation mechanisms combining with Lemma \ref{L32} give the proof of Theorem \ref{T3}.

\begin{lemma}\label{LL33}
Let $(\rho,u,P)$ be the unique local strong solution to the system \eqref{C101}--\eqref{1.2} in $\R\times[0,T]$ given by Lemma \ref{L41}. Then \eqref{C401} holds for a.e. $t\in[0,T]$.
\end{lemma}
\begin{proof}
Let $\varphi_R(x)=\varphi(x/R)$ with $\varphi$ the standard truncation function. For a constant vector $a=(a_1,a_2)\in\R$ and $\nabla^\bot=(-\partial_2,\partial_1)$, define
\[
f(x)=a_2x_1-a_1x_2~~\mathrm{and}~~
\Phi_R=\nabla^\perp(f\cdot\varphi_R)
=\varphi_R\cdot a+f\cdot\nabla^\perp\varphi_R.
\]
We admit that $\|\Phi_R\|_{L^\infty}\leq C|a|$, $\mathrm{supp}\nabla\Phi_{R}\subset B_{2R}\setminus B_{R}$, and $\|\nabla\Phi_R\|_{L^2}\leq C|a|$.

For any $\psi(t)\in C_0^\infty([0,T])$, testing the momentum equation \eqref{C101}$_2$ against $\Phi_R\,\psi$ and applying the fact $\mathrm{div}\Phi_R=0$, 
we obtain that
\[
\int_0^T\int_\R\rho\dot u\cdot\Phi_R\,\psi\,dxdt
=-2\int_0^T\int_{B_{2R}\setminus B_{R}}\mu(\rho)Du\cdot\nabla\Phi_R\,\psi\,dxdt,
\]
and hence the regularity assertion \eqref{1.7} implies that
\[
\left|\int_0^T\int_\R\rho\dot u\cdot\Phi_R\,\psi\,dxdt\right|
\leq C|a|\cdot\int_0^T\|Du\|_{L^2(B_{2R}\setminus B_{R})}dt\rightarrow0
~~\mathrm{as}~~R\rightarrow\infty.
\]
Meanwhile $\Phi_R\to a$ pointwise and is uniformly bounded, moreover \eqref{1.7} ensures $\rho\dot{u}\in L^1(\R\times[0,T])$. 
Thus the dominated convergence theorem yields
\[
a\cdot\int_0^T\left(\int_\R\rho\dot u\,dx\right)\psi(t)\,dt=0.
\]

Since $a\in\R$ and $\psi\in C_0^\infty([0,T])$ are arbitrary, we therefore obtain that $\int_\R\rho\dot u\,dx=0$ for a.e. $t\in[0,T]$
and finish the proof of Lemma \ref{LL33}.
\end{proof}

\begin{lemma}\label{LL34}
Let $(\rho,u,P)$ be the unique local strong solution to the system \eqref{C101}--\eqref{1.2} in $\R\times[0,T]$ given by Lemma \ref{L41}.
Then the total momentum is conserved:
\begin{equation}\label{CC402}
\int_\R\rho u(x,t)\,dx
=\int_\R\rho_0 u_0\,dx
~~\text{for a.e. }t\in(0,T).
\end{equation}
In particular, if the zero total initial momentum condition \eqref{zero} is valid, then $\int_\R\rho u\,dx=0$ for a.e. $t\in[0,T]$.
\end{lemma}
\begin{proof}
Let $a\in\R$ and $\Phi_R$ be the divergence-free cutoff vector field in Lemma \ref{LL33}. For $\psi\in C_0^\infty([0,T))$, testing \eqref{C101}$_2$ against $\Phi_R\,\psi$ and applying $\mathrm{div}\Phi_R=0$,  
we obtain 
\begin{equation}\label{CC403}
\begin{split}
&\int_0^T\psi'(t)\int_\R\rho u\cdot\Phi_R\,dx\,dt+\psi(0)\int_\R\rho_0u_0\cdot\Phi_R\,dx\\
&+\int_0^T\int_\R
\bigl(\rho u\otimes u-2\mu(\rho)Du\bigr)\cdot\nabla\Phi_R\,\psi(t)\,dx\,dt=0.
\end{split}
\end{equation}
Direct computations along with regularity conditions \eqref{1.7} yield that
\begin{equation*}
\begin{split}
&\left|\int_0^T\int_\R
\bigl(\rho u\otimes u-2\mu(\rho)Du\bigr)\cdot\nabla\Phi_R\,\psi(t)\,dx\,dt\right|\\
&\leq\frac{C}{R}
\int_0^T\|\sqrt\rho u\|_{L^2}^2\,dt+ C
\int_0^T\|\nabla u\|_{L^2(B_{2R}\setminus B_R)}\,dt \rightarrow 0 ~~\mathrm{as}~~R\rightarrow\infty.
\end{split}
\end{equation*}
By virtue of the dominated convergence theorem, \eqref{CC403} provides that
\begin{equation}\label{CC404-1}
\int_0^T\psi'(t)\,a\cdot\left(\int_\R\rho u\,dx\right)\,dt+\psi(0)\,a\cdot
\left(\int_\R\rho_0u_0\,dx\right)=0.
\end{equation}
Since $a\in\R$ and $\psi\in C_0^\infty([0,T))$ are arbitrary, \eqref{CC404-1} implies that the total momentum equals its initial value for a.e. $t\in[0,T]$. The proof of Lemma \ref{LL34} is finished.
\end{proof}

With all preparations done, let us turn to prove Theorem \ref{T3}.

{\it Proof of Theorem \ref{T3}:}
The momentum equation \eqref{C101}$_2$ implies that $u$ solves the Stokes system
\begin{equation*}
\begin{cases}
-\Delta u+\nabla P=\mathrm{div}\big(2(\mu(\rho)-1)Du\big)-\rho\dot u,\\
\mathrm{div}u=0.
\end{cases}
\end{equation*}
Moreover, Lemma \ref{LL33} ensures that $\int_\R\rho\dot{u}\,dx=0$ for a.e. $t\in[0,T]$, 
therefore we apply Lemma \ref{L32} and \eqref{1.7} to derive that for 
$0<\beta/2<\beta<1$,
\begin{equation}\label{CC306}
\begin{split}
\|\nabla u\,|x|^{\beta/2}\|_{L^2}
&\leq C_{\beta}\left(\|\rho\|_{L^1}^{1/2}+\|\rho\|_{L^\infty}^{1/2}+\|\rho\|_{L^1}^{(1-\beta)/2}\|\rho\,|x|^2\|_{L^1}^{\beta/2}\right)\,\|\sqrt{\rho}\dot u\|_{L^2}\\
&\leq C(1+M(t)^{\beta/2})\|\sqrt{\rho}\dot{u}\|_{L^2}.
\end{split}
\end{equation}

We define $v=\partial_1\Delta^{-1}(\partial_1u)+\partial_2\Delta^{-1}(\partial_2u)$, and it is direct to check that
$\nabla(u-v)=0$. So there is a constant vector field $u_\infty$ such that
\begin{equation}\label{CC302}
  u=u_\infty+v.
\end{equation}
The weighted Hardy-Littlewood-Sobolev inequality \eqref{C307} and \eqref{CC306} imply that
\begin{equation}\label{CC301}
\|v\|_{L^{4/\beta}}\leq C\|\nabla u\,|x|^{\beta/2}\|_{L^2}
\leq C(1+M(t)^{\beta/2})\|\sqrt{\rho}\dot{u}\|_{L^2}<\infty,
\end{equation}
which provides \eqref{CC101} and \eqref{CC102}. In addition, \eqref{1.7} guarantees $\nabla u\in L^3$,
thus $v\in C^{1/3}$ due to Sobolev embedding. 
Consequently, $v$ is uniformly continuous over $\R$, which along with \eqref{CC301} ensures that $v\rightarrow 0$, or equivalently, $u\rightarrow u_\infty$ as $|x|\rightarrow\infty$. 

Finally, multiplying \eqref{CC302} by $\rho$ and integrating over $\R$ yield
\begin{equation}\label{CC303}
m u_\infty=\int_\R\rho u\,dx-\int_\R\rho v\,dx,~~\mathrm{with}~~m=\int_\R\rho\,dx.
\end{equation}
The conservation of momentum in Lemma \ref{LL34} and \eqref{zero} provide that 
$$\int_\R\rho u\,dx=\int_\R\rho_0u_0\,dx=0,$$
which combined with the formula of $v$ and \eqref{CC303} leads to \eqref{CC99},
\begin{equation}\label{CC304}
u_\infty=\frac{-1}{m}\int_\R\rho v\,dx=\frac{1}{m}\int_\R\nabla\Delta^{-1}\rho\cdot\nabla u\,dx.
\end{equation}
Let us justify the last equality of \eqref{CC304}:
Set $A_0=B_{1}$ and $A_k=B_{2^k}\setminus B_{2^{k-1}}$ with $k\geq 1$. For $p=(4-\beta)/(2-\beta)$ satisfying $p>2$ and $s=1-2/p-\beta/2<0$, we have
\begin{equation*}
\begin{split}
&\int_\R\left|\nabla\Delta^{-1}\rho\cdot\nabla u\right|\,dx\\
&=\sum_{k=0}^\infty\int_{A_k}\left|\nabla\Delta^{-1}\rho\cdot\nabla u\right|\,dx\\
&\leq \int_{B_1}\left|\nabla\Delta^{-1}\rho\cdot\nabla u\right|\,dx+\sum_{k=1}^\infty{2^{-(k-1)\beta/2}}\int_{A_k }|\nabla\Delta^{-1}\rho|\cdot|\nabla u|\cdot|x|^{\beta/2}\,dx\\
&\leq C\left(\|\nabla\Delta^{-1}\rho\|_{L^{4}}\|\nabla u\|_{L^2}+\sum_{k=1}^\infty 2^{(k-1)s}\|\nabla\Delta^{-1}\rho\|_{L^{p}}\|\nabla u\,|x|^{\beta/2}\|_{L^2}\right)\\
&\leq C(\|\rho\|_{L^1}+\|\rho\|_{L^\infty})\,(\|\nabla u\|_{L^2}+\|\nabla u\,|x|^{\beta/2}\|_{L^2})\leq C\|\nabla u(1+|x|)^{\beta/2}\|_{L^2},
\end{split}
\end{equation*}
which ensures the integral converges absolutely and gives \eqref{CC100}. Moreover, using the standard truncation $\varphi_R=\varphi(x/R)$ 
and $\nabla u=\nabla v$ provide that
\begin{equation}\label{CC305}
\begin{split}
-\int_\R\varphi_R\cdot\rho v\,dx=\int_\R\varphi_R\nabla\Delta^{-1}\rho\cdot\nabla u\,dx+\int_\R v\cdot\nabla\Delta^{-1}\rho\cdot \nabla\varphi_R\,dx.
\end{split}
\end{equation}
Note that the ``boundary'' term vanishes due to 
\begin{equation*}
\left|\int_\R v\cdot\nabla\Delta^{-1}\rho\cdot \nabla\varphi_R\,dx\right|\leq C R^s\|\nabla\Delta^{-1}\rho\|_{L^p}\|v\|_{L^{4/\beta}}\rightarrow 0~~\mathrm{as}~~R\rightarrow\infty,
\end{equation*}
which along with \eqref{CC305} gives \eqref{CC304} and finishes the proof of Theorem \ref{T3}.
\thatsall


\section{A priori estimates and exponential decay}\label{S4}
\quad 

In this section, we always assume that $(\rho,u,P)$ is the strong solution to the system \eqref{C101}--\eqref{1.2} on $\R\times(0,T]$ guaranteed by Lemma \ref{L41}, and the constants $\alpha,\beta$ always satisfy that $0<\alpha<\beta<1/2$.

\subsection{Lower order estimates}\label{S41}
\quad  
The arguments in this subsection are established for general density-dependent viscosity models, including both the settings of Theorems \ref{T1} and \ref{T2}. First of all, it is well-known that the $L^p$-norm $(1\le p\le\infty)$ of $\rho$ is preserved by the incompressible flow (see \cite[Lemma 3.4]{lions}).
\begin{lemma}\label{le1}
Under the conditions of Theorem \ref{T1}, we have 
\begin{equation}\begin{aligned}\label{3.7} 
\|\rho(\cdot,t)\|_{L^\infty}=\|\rho_0\|_{L^\infty}\leq\bar\rho,\quad 
\|\rho(\cdot,t)\|_{L^1}=\|\rho_0\|_{L^1},\quad t\in (0,T].
\end{aligned}\end{equation}
In particular, there are positive constants   $\underline{\mu}$ and $\overline{\mu}$ determined by  $\bar\rho$ and $\mu$  such that 
\begin{equation}\label{1.5}
\begin{split}
0<\underline{\mu}\leq\mu(\rho(x,t))\leq\overline{\mu}<\infty,\quad (x,t)\in\R\times[0,T].\\
\end{split}
\end{equation}
Meanwhile, it holds that
\begin{equation}\label{C415}
\|\mu(\rho(\cdot,t))-1\|_{L^\infty}=\|\mu(\rho_0)-1\|_{L^\infty}, \quad t\in(0,T].
\end{equation}
\end{lemma}

Next, we present a series of  basic energy estimates.
\begin{lemma}\label{le2} 
Under the conditions of Theorem \ref{T1}, there exists a positive constant $C$ determined by  $\bar\rho$, $\mu$, and $\|\sqrt{\rho_0}u_0\|_{L^2}$ such that
\begin{equation}\label{3.8} 
\sup_{t\in[0,T]}\|\sqrt\rho u\|_{L^2}^2+\int_0^T\|\nabla u\|_{L^2}^2\,dt\leq C.
\end{equation}
\end{lemma}
\begin{proof}
Taking the $L^2$ inner product of  \eqref{C101}$_2$ with $u$, we get
\begin{equation}\label{3.12}
\frac{1}{2}\,\frac{d}{dt}\int_\R\rho|u|^2\,dx+2\int_\R\mu(\rho)|Du|^2\,dx=0.
\end{equation}
Note that $2\int_\R |Du|^2\,dx=\int_\R|\nabla u|^2\,dx$, so  this along with \eqref{1.5} and \eqref{3.12} yields   \eqref{3.8}. The proof of Lemma \ref{le2} is finished.
\end{proof}

\begin{lemma}\label{le3}
Under the conditions of Theorem \ref{T1}, there exist  positive constants  $\varepsilon_1$ and $C$ depending on $\bar\rho$, $\mu$, $\|\sqrt{\rho_0}u_0\|_{L^2}$,  and $\|\nabla u_0\|_{L^2}$  such that if 
$$\|\mu(\rho_0)-1\|_{L^\infty}\leq \varepsilon_1,$$
it holds that
\begin{equation}\label{3.9}
\sup_{t\in[0,T]}\|\nabla u\|_{L^2}^2+\int_0^T \|\sqrt\rho\dot u\|_{L^2}^2\,dt\leq C.
\end{equation}
\end{lemma}
\begin{proof}
Multiplying \eqref{C101}$_2$ by $\dot{u}$, we obtain after integrating by parts that
\begin{equation}\label{3.6}
2\int_\R\mu(\rho)Du:D\dot u\,dx+\int_\R\rho|\dot{u}|^{2}\,dx=\int_\R P\,\mathrm{div}\dot{u}\,dx=\int_\R P\,\partial_{i}u\cdot\nabla u^{i}\,dx.
\end{equation}
The last equality is due to $\mathrm{div}\dot u=\mathrm{div}u_t+u\cdot\nabla\mathrm{div}u+\partial_{i}u\cdot\nabla u^{i}=\partial_{i}u\cdot\nabla u^{i}$.

To estimate the first term in \eqref{3.6}, we check that
\begin{equation}
\begin{split}\label{3.28}
&2\int_\R\mu(\rho) Du:D\dot u\, dx\\ 
&=2\int_\R\mu(\rho)(Du:D u_t+Du: (u\cdot\nabla Du))\,dx+2\int_\R\mu(\rho)\,\partial_iu^j\,\partial_i u\cdot\nabla u^j\,dx\\
&\geq\frac{d}{dt}\int_\R\mu(\rho)|Du|^2\,dx-\int_\R|Du|^2\left(\partial_t\mu(\rho)+u\cdot\nabla\mu(\rho)\right)\,dx-C\int_\R|\nabla u|^3\,dx\\
&=\frac{d}{dt}\int_\R\mu(\rho)|Du|^2\,dx-C\int_\R|\nabla u|^3\,dx,
\end{split}
\end{equation}
where we have applied \eqref{1.5} and 
$\partial_t\mu(\rho)+u\cdot \nabla \mu(\rho)=0$
owing to \eqref{C101}$_1$ and $\mathrm{div}u=0$.  

Observe that  $P$ can be obtained via
\begin{equation}\label{C225}
P=-\Delta^{-1}\mathrm{div}(\rho\dot u)+\Delta^{-1}\mathrm{div}\mathrm{div}(2\mu(\rho) Du).
\end{equation}
Then in view of \eqref{C225}, \eqref{1.5}, and the fact:  
``$\Delta^{-1}\mathrm{div}\mathrm{div}$" maps $L^3(\R)$ to itself (see \cite[Page 13]{S}), the right hand side of \eqref{3.6} satisfies
\begin{equation}\label{C416}
\begin{split}
&\int_\R P\,\partial_{i}u\cdot\nabla u^{i}\,dx\\
&=-\int_\R\Delta^{-1}\mathrm{div}(\rho\dot u)\,\partial_{i}u\cdot\nabla u^{i}\,dx+\int_\R \Delta^{-1}\mathrm{div}\mathrm{div}(2\mu(\rho)Du)\,\partial_{i}u\cdot\nabla u^i\,dx\\
&\leq -\int_\R\Delta^{-1}\mathrm{div}(\rho\dot u)\,\partial_{i}u\cdot\nabla u^{i}\,dx+C\int_\R|\nabla u|^3\,dx.
\end{split}
\end{equation}
Moreover,
we adopt the idea from \cite{des} to argue that
\begin{equation}\begin{aligned}\label{3.30}
\left|\int_\R\Delta^{-1}\mathrm{div}(\rho\dot u)\,\partial_{i}u\cdot\nabla u^{i}\,dx\right|
&\leq \|\Delta^{-1}\mathrm{div}(\rho\dot u)\|_{\mathrm{BMO(\R)}}\|\partial_{i}u\cdot\nabla u^{i}\|_{\mathcal{H}^{1}(\R)}\\
&\leq C\,\|\nabla\Delta^{-1}\mathrm{div}(\rho\dot u)\|_{L^2}\|\nabla u \|_{L^2}^{2}\\
&\leq C\,\|\rho\dot u\|_{L^2}\|\nabla u\|_{L^2}^2,
\end{aligned}\end{equation}
due to BMO$(\mathbb{R}^2)$ being the dual of $\mathcal{H}^1(\mathbb{R}^2)$, and the div-curl lemma (\cite[Theorem II.1]{clms}): for $\mathrm{div}(\partial_i u)=0$ and $ \mathrm{curl}(\nabla u^i)=0,$ there is a constant $C$ such that
\begin{equation}\notag
\|\partial_{i}u\cdot\nabla u^{i}\|_{\mathcal{H}^1(\R)}\leq C\,\|\nabla u\|_{L^2}^2.
\end{equation}

Substituting \eqref{3.28}, \eqref{C416}, \eqref{3.30} into \eqref{3.6} and using Lemma \ref{pl3},  \eqref{3.7}--\eqref{C415}, there exists a positive constant $\varepsilon_1$ such that if $\|\mu(\rho_0)-1\|_{L^\infty}\le \varepsilon_1$,
\begin{equation}\begin{aligned}\label{3.13}
\frac{d}{dt}\int_\R \mu(\rho)|Du|^2\, dx+\int_\R\rho|\dot u|^2\,dx
&\leq C\left(\|\rho\dot u\|_{L^2}\|\nabla u\|_{L^2}^2+\|\nabla u\|_{L^3}^3\right)\\
&\leq C\,\|\rho\dot u\|_{L^2}\|\nabla u\|_{L^2}^2\\
&\leq \frac{1}{2}\int_\R\rho|\dot u|^2\,dx+C\,\|\nabla u\|_{L^2}^4.
\end{aligned}\end{equation}
Applying Gronwall's inequality to the above inequality  along with \eqref{1.5} and  \eqref{3.8}  leads to \eqref{3.9}. The proof of Lemma \ref{le3} is completed.
\end{proof}

\begin{lemma}\label{le4}
Under the conditions of Theorem \ref{T1}, there exists  positive constants $\varepsilon_2$ $(\le \varepsilon_1)$  and 
 $C$ depending on $\bar\rho$, $\mu$, $\|\sqrt{\rho_0}u_0\|_{L^2}$, $\|\nabla u_0\|_{L^2}$ such that if 
$$\|\mu(\rho_0)-1\|_{L^\infty}\leq \varepsilon_2,$$
it holds that  for $\delta(t)\triangleq\min\{1,t\}$,
\begin{equation}\label{3.10}
\sup_{t\in[0,T]}\delta(t)\,\|\sqrt\rho\dot u\|_{L^2}^2+\int_0^T \delta(t)\,\|\nabla\dot u\|_{L^2}^2\,dt\leq C.
\end{equation}
\end{lemma}
\begin{proof}
First, applying the operator $\partial_t+
u\cdot\nabla $ to  \eqref{C101}$_2$, we obtain that
\begin{equation}\begin{aligned}\notag
&\rho\dot u_t+\rho u\cdot\nabla\dot u-\mathrm{div}(2\mu(\rho)D\dot u)+\partial_i\left(\mu(\rho)\partial_i u\cdot\nabla u\right)+\partial_i\left(\mu(\rho)\nabla u^i\cdot\nabla u\right)\\
&+2\,\partial_i\left(\mu(\rho)\nabla u\cdot\nabla u^i\right)+\nabla P_t+u\cdot\nabla\left(\nabla P\right)=0.
\end{aligned}\end{equation}
Taking the $L^2$ inner product of this equality with $\dot u$ and using $\mathrm{div}u=0$ and \eqref{1.5} yield
\begin{equation}\label{3.34}\begin{aligned}
&\frac{1}{2}\,\frac{d}{dt}\int_\R\rho|\dot{u}|^{2}dx+2\int_\R\mu(\rho)\,|D\dot u|^2\,dx\\
&=\int_\R\mu(\rho)\left(\partial_i u\cdot\nabla u\cdot\partial_i\dot u+\nabla u^i\cdot\nabla u\cdot\partial_i\dot u+2\,\partial_i\dot u\cdot\nabla u\cdot\nabla u^i\right)\,dx\\
&\quad+\int_\R \left(P_t+u\cdot\nabla P\right)\mathrm{div}\dot u\,dx+\int_\R \dot u\cdot\nabla u\cdot\nabla P\,dx\\
&\leq C\int_\R |\nabla u|^2|\nabla \dot u|\,dx+\int_\R \left(P_t+u\cdot\nabla P\right)\partial_i u\cdot\nabla u^i\,dx-\int_\R P\,\partial_i u\cdot\nabla  \dot u^i\,dx.
\end{aligned}
\end{equation}

Note that integration by parts combined with $\mathrm{div} u=0$ shows that
\begin{equation}\begin{aligned}
&\int_\R \left(P_t+u\cdot\nabla P\right)\partial_i u\cdot\nabla u^i\,dx\\
&=\frac{d}{dt}\int_\R P \,\partial_i u\cdot\nabla u^i\, dx-2\int_\R P\,\partial_i u\cdot\nabla \left(\dot u^i-u\cdot\nabla u^i\right)\,dx\\
&\quad-\int_\R P  u\cdot\nabla\left(\partial_i u\cdot\nabla u^i\right)\,dx\\
&=\frac{d}{dt}\int_\R P \,\partial_i u\cdot\nabla u^i\, dx-2\int_\R P\, \partial_i u\cdot\nabla \dot u^i\,dx+2\int_\R P\, \partial_i u\cdot\nabla u\cdot\nabla u^i\,dx\\
&\quad +2\int_\R P\, u\cdot\nabla \partial_i u\cdot\nabla u^i\,dx-\int_\R P\, u\cdot\nabla\left(\partial_i u\cdot\nabla u^i\right)\,dx\\
&=\frac{d}{dt}\int_\R P\, \partial_i u\cdot\nabla u^i\,dx-2\int_\R P\, \partial_i u\cdot\nabla \dot u^i\,dx+2\int_\R P\,\partial_i u\cdot\nabla u\cdot\nabla u^i\,dx.
\end{aligned}
\end{equation}
Moreover in view of $\mathrm{div}\dot u=\partial_iu\cdot\nabla u^i$, we also have
\begin{equation}\begin{aligned}\label{3.36}
2\int_\R|D\dot u|^2\,dx 
&=\int_\R\left(|\nabla \dot u|^2+(\mathrm{div}\dot u)^2\right)\,dx 
\geq\int_\R|\nabla \dot u|^2\,dx-\int_\R|\nabla u|^4\,dx.
\end{aligned}\end{equation}

Consequently, it follows from \eqref{3.34}--\eqref{3.36}, \eqref{1.5}, and Young's inequality that
\begin{equation}\label{3.37}\begin{aligned}
&\frac{1}{2}\,\frac{d}{dt}\int_\R\rho|\dot{ u }|^{2}\,dx+\frac{1}{2}\underline\mu\int_\R|\nabla \dot u|^2\,dx-\frac{d}{dt}\int_\R P\,\partial_i u\cdot\nabla u^i\,dx\\
&\leq C\int_\R|\nabla u|^4\,dx-3\int_\R P\,\partial_i  u\cdot\nabla  \dot u^i\,dx+2\int_\R P\, \partial_i u\cdot\nabla u\cdot\nabla u^i\,dx\\
&\leq \frac{1}{4}\underline\mu\int_\R|\nabla \dot u|^2\,dx+C\left(\int_\R|P|^4\,dx+\int_\R|\nabla u|^4\,dx\right)\\
&\leq \frac{1}{4}\underline\mu\int_\R|\nabla \dot u|^2\,dx+C\left(\|\sqrt{\rho}\dot u\|_{L^2}^4+\|\nabla u\|_{L^2}^4\right).
\end{aligned}
\end{equation}
In the last line we used  the fact: for some constant $\varepsilon_2\,(\le \varepsilon_1)$, if $\|\mu(\rho_0)-1\|_{L^\infty}\le \varepsilon_2$,
\begin{equation}\begin{aligned}\label{4}\notag
\|\nabla u\|_{L^4}+\|P\|_{L^4}&\leq C\left(\|\nabla u\|_{L^4}+\|\Delta^{-1}\mathrm{div}(\rho\dot u)\|_{L^4}\right)\\
&\leq C\left(\|\rho\dot u\|_{L^{2}}+\|\nabla u\|_{L^2}+\|\rho\dot u\|_{L^{4/3}}\right)\\
&\leq C\left(\|\sqrt\rho\dot u\|_{L^{2}}+\|\nabla u\|_{L^2}\right),
\end{aligned}\end{equation}
due to Lemma \ref{pl3}, \eqref{C225}, and \eqref{3.7}--\eqref{C415}.

In addition, combining \eqref{C218}, \eqref{C416}, \eqref{3.30},  \eqref{3.7}--\eqref{C415}, and \eqref{3.9}  leads to that 
\begin{equation}\label{3.38}\begin{aligned}
&\left|\int_\R P\,\partial_i u\cdot\nabla u^i\,dx\right|\\
&\leq C\left(\|\rho\dot u\|_{L^{2}}\|\nabla u\|_{L^2}^2+\|\nabla u\|_{L^3}^3\right)
\leq \frac{1}{4}\|\sqrt\rho\dot u\|_{L^{2}}^2+C\|\nabla u\|_{L^2}^4.
\end{aligned}
\end{equation}

Multiplying \eqref{3.37} by $\delta(t)$ and applying Gronwall's equality to the resulting inequality, we derive from \eqref{3.8}, \eqref{3.9}, and \eqref{3.38}  that \eqref{3.10} is valid provided $\|\mu(\rho_0)-1\|_{L^\infty}\le \varepsilon_2$. The proof of Lemma \ref{le4} is therefore completed.
\end{proof}

The next lemma provides the space weighted estimate (not uniform with respect to $t$) on $\rho$,  see also \cite[Lemma 3.5]{lx}.
\begin{lemma}\label{le9}
Under the conditions of Theorem \ref{T1}, there exists a positive constant $C$ depending on $\|\rho_0\,(1+|x|^2)\|_{L^1}$ and $\|\sqrt{\rho_0}u_0\|_{L^2}$ such that for any $t\in(0,T]$,
\begin{equation}\label{k37} 
\int_\R \rho(x,t)\,(1+|x|^2)\,dx\leq C(1+t)^2.
\end{equation}
In particular, there is a positive constant $N$ depending only on $\|\rho_0\|_{L^1}$, $\|\rho_0\,(1+|x|^2)\|_{L^1}$, and $\|\sqrt{\rho_0}u_0\|_{L^2}$ such that for any $t\in (0,T]$, 
\begin{equation}\label{k38} 
\int_{B_{N(1+t)}}\rho(x,t)\,dx\geq \frac{1}{2}\,\|\rho_0\|_{L^1}.
\end{equation}
\end{lemma}
\begin{proof}
Multiplying \eqref{C101}$_1$ by $(1+|x|^2)$  shows that
\begin{equation}\label{C417}
\frac{d}{dt}\int_\R \rho\,(1+|x|^2)\, dx \leq 2\int _\R\rho|u|\cdot|x|\,dx\leq 2\,\|\sqrt\rho u\|_{L^2}\left(\int_\R\rho\,(1+|x|^2)\,dx\right)^{{1}/{2}}.
\end{equation}
Integrating \eqref{C417} over $(0,t)$ and using \eqref{3.8} lead to \eqref{k37} directly.

 Next,  setting $N^2=2C/\|\rho_0\|_{L^1}$ with $C$ in \eqref{k37}, we get from \eqref{k37} and \eqref{3.7} that
\begin{equation*}
\int_{\{|x|\geq  N(1+t)\}}\rho(x,t)\,dx\leq\frac{1}{N^2(1+t)^2}\int_{\{|x|\geq  N(1+t)\}}\rho(x,t)\,|x|^2\,dx\leq \frac{1}{2}\,\|\rho\|_{L^1},
\end{equation*}
which combined with \eqref{3.7} gives \eqref{k38} and finishes the proof of Lemma \ref{le9}.
\end{proof}
 
Now we declare the properties of exponential decay to the system \eqref{C101}--\eqref{1.2} in $\R$.

\begin{lemma}\label{le5} 
Under the conditions of Theorem \ref{T1}, suppose that
$$\|\mu(\rho_0)-1\|_{L^\infty}\leq \varepsilon_3,$$
with $\varepsilon_3\triangleq\min\{\varepsilon_0,\varepsilon_2\}$, where $\varepsilon_0$ and $\varepsilon_2$ are given by Lemmas \ref{L32} and \ref{le4} respectively.
There are positive   constants $C$ and $\sigma$ both depending on $\bar\rho$, $\mu$,  $\|\rho_0\,(1+|x|^2)\|_{L^1}$, $\|\sqrt{\rho_0}u_0\|_{L^2}$, and $\|\nabla u_0\|_{L^2}$ such that
\begin{equation}\begin{aligned}\label{3.11}
&\sup_{t\in[0,T]}\left(e^{\sigma t}\left(\|\sqrt\rho u\|_{L^2}+\| \nabla u\|_{L^{2} }+\delta(t)\,\|\sqrt\rho\dot u\|_{L^2}^2\right)\right)+\sup_{t\in[0,T]}\|\rho\,|x|^2\|_{L^1}\\
&+\int_0^Te^{\sigma t}\left(\| \nabla u\|_{L^{2} }^2+\|\sqrt\rho\dot u\|_{L^2}^2+\delta(t)\,\|\nabla \dot u\|_{L^2}^2\right)\,dt\leq C.
\end{aligned}\end{equation} 
\end{lemma}
\begin{proof}
First, according to \eqref{1.5}, \eqref{3.12}, \eqref{3.9}, and \eqref{3.13}, we can find a small constant $C_1\,(<1)$ such that 
\begin{equation}\label{3.14}
\frac{d}{dt}\int_\R\left(\rho|u|^2+C_1\mu(\rho)|Du|^2\right)\,dx+\int_\R\left(\mu(\rho)|D u|^2+\frac{C_1}{2}\,\rho|\dot u|^2\right)\,dx\leq 0.
\end{equation}
In view of Theorem \ref{T3}, we first select $\beta_0\in(0,1/2)$ and $\|\mu(\rho_0)-1\|_{L^\infty}\le\varepsilon_0$ with $\varepsilon_0$ given in \eqref{CC90}. 
For the far field velocity $u_{\infty}(t)$, it holds that
\begin{equation}\label{C425}
\begin{split}
&\left(\int_\R|u-u_{\infty}(t)|^{4/\beta_0}\,dx\right)^{\beta_0/4}\leq  C\left(1+\|\rho\,|x|^2\|_{L^1}^{\beta_0/2}\right)\|\sqrt{\rho}\dot u\|_{L^2},
\end{split}
\end{equation}
which combined with  \eqref{CC402}, \eqref{3.7}, and \eqref{k37} provides that for some constant $C_2\,(>1)$,
\begin{equation}\label{3.15}
\begin{split}
\int_\R\rho |u|^2\,dx
&=\int_\R\rho |u-u_{\infty}(t)|^2\,dx-\int_\R\rho |u_{\infty}(t)|^2\,dx
\\
&\leq\int_\R\rho|u-u_{\infty}(t)|^2\,dx\\
&\leq \|\rho\|_{L^{2/(2-\beta_0)}}\|u-u_\infty(t)\|_{L^{4/\beta_0}}^2\leq C_2\,(1+t)^{2\beta_0} \|\sqrt\rho\dot u\|_{L^2}^2.
\end{split}
\end{equation}
Collecting \eqref{1.5}, \eqref{3.14}, and \eqref{3.15} shows that 
\begin{equation*} 
\frac{d}{dt}\int_\R\left(\rho|u|^2+C_1\,\mu(\rho)|Du|^2\right)\,dx+C_3\,(1+t)^{-2\beta_0}\int_\R\left(\rho|u|^2+C_1\,\mu(\rho)|Du|^2\right)\,dx\le 0,
\end{equation*}
with $C_3=C_1/(2C_2)$.
Thus for $\sigma\triangleq C_3/(2-4\beta_0)$, it holds that for any $t\in(0,T]$,
\begin{equation}\label{3.16}
\int_\R\left(\rho|u|^2+C_1\,\mu(\rho)|Du|^2\right)\,dx\leq C\,e^{-2\sigma(1+t)^{1-2\beta_0}}.
\end{equation}
In particular, \eqref{C417} and \eqref{3.16} lead to
\begin{equation}\label{C426}
\sup_{0\leq t\leq T}\int_\R\rho(x,t)\,(1+|x|^2)\,dx\leq \left(\|\rho_0\,(1+|x|^2)\|_{L^1}^{1/2}+2\int_0^T\|\sqrt\rho u\|_{L^2}\,dt\right)^{2}\leq C.
\end{equation}

Then let us substitute \eqref{3.7} and \eqref{C426} back to \eqref{C425} and \eqref{3.15} to find that
\begin{equation}\label{C427}
\int_\R\rho |u|^2\,dx\leq C \|\sqrt\rho\dot u\|_{L^2}^2.
\end{equation}
Consequently, \eqref{C427} in fact implies that
\begin{equation*} 
\frac{d}{dt}\int_\R\left(\rho|u|^2+C_1\,\mu(\rho)|Du|^2\right)\,dx+C_4\int_\R\left(\rho|u|^2+C_1\,\mu(\rho)|Du|^2\right)\,dx\le 0,
\end{equation*}
which along with the Gronwall's inequality leads to
\begin{equation} \label{CC404}
\int_\R\left(\rho|u|^2+C_1\,\mu(\rho)|Du|^2\right)\,dx\leq C e^{-C_4 t}.
\end{equation}
Moreover, we deduce from \eqref{3.14}, \eqref{CC404}, and \eqref{1.5} that for some constant $\sigma$,
\begin{equation}\label{3.17}
 \int_0^Te^{\sigma(1+t)}\left(\|\nabla u\|^2_{L^2}+\|\sqrt\rho \dot u\|^2_{L^2}\right)\,dt\leq C.
\end{equation}
By \eqref{3.37}, \eqref{3.38}, \eqref{3.17},  \eqref{3.9}, and Gronwall's inequality,  we thus get that
\begin{equation}\label{3.18}
\sup_{t\in[0,T]}\left(\delta(t)\,e^{\sigma(1+t)}\|\sqrt\rho \dot u\|_{L^2}^2\right)
+\int_{0}^T\delta(t)\,e^{\sigma(1+t)}\|\nabla \dot u\|_{L^2}^2\,dt\leq C,
\end{equation}
if $\|\mu(\rho_0)-1\|_{L^\infty}\le\varepsilon_3\triangleq\min\{\varepsilon_0,\varepsilon_2\}$. 
Finally, combining \eqref{CC404}--\eqref{3.18} with \eqref{1.5}  gives \eqref{3.11} and finishes the  proof of Lemma \ref{le5}.
\end{proof}

\subsection{Higher order estimates}
\quad
Note that the estimates in Subsection \ref{S41} do not involve spatial derivatives of $\rho$.  
If we impose a stronger regularity: $H^1\cap W^{1,q}$ with $q>2$ for $\rho$ in Theorem \ref{T1}, or $\mu(\rho)\equiv 1$ in Theorem \ref{T2}, further uniform higher order estimates are in orders, 
which lead to the global strong solutions and uniform regularities of flow maps. Here we present the higher order estimates for the density-dependent viscosity model, and the corresponding ones for the case $\mu(\rho)\equiv 1$ are simpler, see Remark \ref{R43}.

We first show the uniform $L^1(0,T;L^\infty)$-norm of $\nabla u$.
\begin{lemma}\label{le6}
Under the conditions of Theorem \ref{T1}, there exist  positive constants $\varepsilon_4$ $(\le \varepsilon_3)$  and $C$ depending on  $\bar\rho$, $\mu$, $q$, 
and the initial values such that if
$$\|\mu(\rho_0)-1\|_{L^\infty}\leq\varepsilon_4,$$
 it holds that
\begin{equation}\label{3.19}
\sup_{0\leq t\leq T}\|\rho\|_{H^1\cap W^{1,q}}+\int_0^T\|\nabla u\|_{L^\infty}\,dt\leq C.
\end{equation}
\end{lemma}
\begin{proof}
For simplicity, we assume $q=4$. First, differentiating  \eqref{C101}$_1$ gives 
\begin{equation}\notag
\left(\nabla \rho\right)_t + u \cdot \nabla (\nabla\rho) + \nabla u \cdot \nabla  \rho=0.
\end{equation}
Multiplying this by $\nabla\rho\,|\nabla\rho|^{p-2}$ with $p\in [2,q]$ and integrating by parts lead to 
\begin{equation}\label{3.39}
\frac{d}{dt}\|\nabla \rho \|_{L^p}\leq \|\nabla \rho\|_{L^p}\|\nabla u\|_{L^\infty}.
\end{equation}

Recalling \eqref{du} in Lemma \ref{L32}, 
\begin{equation}\label{C422}
\begin{split}
\nabla u
&=-\nabla\Delta^{-1}\mathbb{P}\,\mathrm{div}\left(2(\mu(\rho)-1)Du\right)-\nabla\Delta^{-1}\mathbb{P}(\rho\dot u)\\
&=2\int_\R {K}_2(x-y)\, \mathrm{div}\big((\mu(\rho)-1)Du\big)(y)\,dy-\nabla\Delta^{-1}\mathbb{P}(\rho\dot u) \\
&\triangleq \mathbf{Q}_1(x)+\mathbf{Q}_2(x).
\end{split}
\end{equation}
where $K_2$ satisfies \eqref{k2}  and  $\mathbb{P}=\mathrm{Id}-\nabla\Delta^{-1}\mathrm{div}$ is the Leray projector.
 
In order to bound $\mathbf{Q}_1(x)$, we utilize the truncation function $\varphi\in C_0^\infty(\R)$ defined by  \eqref{var} again 
and set $\varphi_\lambda(\cdot)=\varphi\big(\cdot/\lambda\big)$ for some $\lambda=\lambda(t)$ determined later, then we write 
\begin{equation}
\begin{split}\label{Q1}
\mathbf{Q}_1(x)
&=2\int_\R\varphi_\lambda(x-y) K_2(x-y)\,\mathrm{div}\left((\mu(\rho)-1)Du\right)(y)\,dy\\
&\quad+2\int_\R\left(1-\varphi_\lambda(x-y)\right) K_2(x-y)\,\mathrm{div}\left((\mu(\rho)-1)Du\right)(y)\,dy\\
&\triangleq \mathbf{Q}_{11}(x)+\mathbf{Q}_{12}(x).
\end{split}
\end{equation}
From \eqref{k2}, we have 
$|\nabla^m K_2(x)|\leq C(m)|x|^{-(m+1)}$ for any  $m\in\mathbb{N}$ and $x\in\R\setminus\{0\}$, this jointed to \eqref{C415} yields  that
\begin{equation}\label{C420}
\begin{split}
|\mathbf{Q}_{11}(x)|&\leq C\int_{\{|x-y|\leq2\lambda\}}|x-y|^{-1}\left|\nabla\mu(\rho)\cdot Du+(\mu(\rho)-1)\cdot\nabla^2 u\right|(y)\,dy\\
&\leq C\,\lambda^{1/3}\left(\|\nabla\mu(\rho)\cdot Du\|_{L^3}+\|\nabla^2u\|_{L^3}\right)\\
&\le C\,\lambda^{1/3}\left(\|\nabla\rho\|_{L^4}+1\right)\left(\|\nabla u\|_{L^2}+\|\rho\dot u\|_{L^2}+\|\rho\dot u\|_{L^3}\right).
\end{split}
\end{equation}
In the last line, we have applied Lemma \ref{pl3} to check that if $\|\mu(\rho_0)-1\|_{L^\infty}\le \varepsilon_{41}$ for some small $\varepsilon_{41}$,
\begin{equation*}
\begin{split} 
&\|\nabla^2 u\|_{L^3}
\leq C\left(\|\rho\dot u\|_{L^3} +\|\nabla\mu(\rho)\cdot Du\|_{L^3}\right),\\
&\|\nabla\mu(\rho)\cdot Du\|_{L^3}\leq C\,\|\nabla\rho\|_{L^4}\|\nabla u\|_{L^{12}}\leq C\,\|\nabla\rho\|_{L^4}\left(\|\nabla u\|_{L^2}+\|\rho\dot u\|_{L^2}\right).\\
\end{split}
\end{equation*}
Meanwhile, integrating by parts indicates that
\begin{equation}
\begin{split}
|\mathbf{Q}_{12}(x)| 
&\leq C\,\|\mu(\rho)-1\|_{L^\infty}\int_\R(|\nabla\varphi_\lambda||x-y|^{-1}+(1-\varphi_\lambda)|x-y|^{-2})|D u|(y)\,dy\\
& \leq C_4\,\|\mu(\rho)-1\|_{L^\infty}\left(\|\nabla u\|_{L^\infty}+\|\nabla u\|_{L^2}\||x-y|^{-2}\|_{L^2(\{|x-y|\geq \lambda\})}\right)\\
&\leq \frac{1}{2}\|\nabla u\|_{L^\infty}+C\,\varepsilon_4\,\lambda^{-1}\|\nabla u\|_{L^2},
\end{split}
\end{equation}
provided that $\|\mu(\rho_0)-1\|_{L^\infty}\leq \varepsilon_4\le \min\{\varepsilon_{41},(2C_4)^{-1}\}$. 
 
For $\mathbf{Q}_2(x)=-\nabla\Delta^{-1}\mathbb{P}(\rho\dot u)$, 
 the Sobolev embedding implies that
\begin{equation}\label{C419}
\begin{split}
|\mathbf{Q}_2(x)|&\leq C\left(\|\mathbf{Q}_2\|_{L^3}+\|\nabla{\mathbf{Q}}_2\|_{L^3}\right)\\
&\le C\left(\big\|\nabla^2\Delta^{-1}\mathbb{P}(\rho\dot u)\big\|_{L^{6/5}}+\big\|\nabla^2\Delta^{-1}\mathbb{P}(\rho\dot u)\big\|_{L^3}\right)\\
&\leq C\left(\|\rho\dot u\|_{L^{6/5}}+\|\rho\dot u\|_{L^3}\right),
\end{split}
\end{equation}
since the operator ``$\nabla^2\Delta^{-1}\mathbb{P}$" is bounded in $L^p$ $(1<p<\infty)$  (see \cite[Chapter 1]{S}). 

Putting \eqref{Q1}--\eqref{C419} into \eqref{C422}, then we arrive at that if $\|\mu(\rho_0)-1\|_{L^\infty}\leq\varepsilon_4$,
\begin{equation}\label{C423}
\begin{split}
\|\nabla u\|_{L^\infty}
&\leq  C\,\lambda^{1/3}\left(\|\nabla\rho\|_{L^4}+1\right)\left(\|\nabla u\|_{L^2}+\|\rho\dot u\|_{L^2}+\|\rho\dot u\|_{L^3}\right)\\
&\quad+C\,\varepsilon_4\,\lambda^{-1}\|\nabla u\|_{L^2}+C\left(\|\rho\dot u\|_{L^{6/5}}+\|\rho\dot u\|_{L^3}\right)\\
&\leq C\left(\varepsilon_4^{1/4}\,\|\nabla\rho\|_{L^4}^{3/4}+1\right)\left(\|\nabla u\|_{L^2}+\|\sqrt\rho\dot u\|_{L^2}+\|\rho\dot u\|_{L^3}\right),
\end{split}
\end{equation}
by using \eqref{3.7} and choosing $\lambda=\lambda(t)=\varepsilon_4^{3/4}\left(1+\|\nabla\rho\|_{L^4}\right)^{-3/4}$. 

Combining \eqref{3.39} and \eqref{C423}, we therefore argue that
\begin{equation}\label{C424}
\frac{d}{dt}\|\nabla\rho\|_{L^4}
\leq C_5\left(\varepsilon^{1/4}\,\|\nabla\rho\|_{L^4}^{7/4}+\|\nabla\rho\|_{L^4}\right)\left(\|\nabla u\|_{L^2}+\|\sqrt\rho\dot u\|_{L^2}+\|\rho\dot u\|_{L^3}\right).
\end{equation}
Notice that \eqref{2.1} as well as \eqref{3.7}, \eqref{k38}, and \eqref{k37} indicates for any $p\in[2,\infty)$,
\begin{equation}\begin{aligned}\label{3.23}
\|\rho\dot u\|_{L^p} 
&\le C(p)\left(N(1+t)\right)^3\left(1+\|\rho\,(1+|x|^2)\|_{L^1}\right) \left(\|\sqrt\rho \dot u\|_{L^2}+\|\nabla\dot u\|_{L^2}\right)\\
&\le C(p)\left(1+t\right)^5\left(\|\sqrt\rho \dot u\|_{L^2}+\|\nabla\dot u\|_{L^2}\right).
\end{aligned}\end{equation}
Hence, it follows from \eqref{3.8}, \eqref{3.9},   \eqref{3.10}, \eqref{3.7}, and \eqref{3.23} that
\begin{equation}\begin{aligned}\label{3.22}
&\int_0^{\delta(T)}\left(\|\nabla u\|_{L^2}+\|\sqrt\rho\dot u\|_{L^2}+\|\rho\dot u\|_{L^3}\right)\,dt\\
&\leq C\int_0^{\delta(T)}\left(\|\nabla u\|_{L^2}+\|\sqrt\rho\dot u\|_{L^2}+\|\rho\dot u\|_{L^2}^{1/2}\,\left(t^{1/2}\|\rho\dot u\|_{L^{6}}\right)^{1/2}\,t^{-1/4}\right)\,dt\\
&\leq C\int_0^{\delta(T)}\left(1+t^{-1/2}+\|\nabla u\|_{L^2}^2+\|\sqrt{\rho}\dot u\|_{L^2}^2+t\|\nabla\dot u\|_{L^2}^2\right)\,dt\leq C_6.
\end{aligned}\end{equation}
Moreover, we deduce from \eqref{3.11} and \eqref{3.23} that  if $\|\mu(\rho_0)-1\|_{L^\infty}\leq\varepsilon_{3}$, 
\begin{equation}\begin{aligned}\label{3.24}
&\int_{\delta(T)}^T\left(\|\nabla u\|_{L^2}+\|\sqrt\rho\dot u\|_{L^2}+\|\rho\dot u\|_{L^3}\right)\,dt\\
&\le C\int_{\delta(T)}^T(1+t)^5\left(\|\nabla u\|_{L^2}+\|\sqrt\rho\dot u\|_{L^2}+\|\nabla\dot u\|_{L^2}\right)\,dt\\
&\leq C \int_{\delta(T)}^T\left(t^{10}e^{-\sigma t}+ e^{\sigma t}\big(\|\nabla u\|_{L^2}^2+\|\sqrt\rho\dot u\|_{L^2}^2+\|\nabla\dot u\|_{L^2}^2\big)\right)\,dt  \leq C_7.
\end{aligned}\end{equation}
Finally, applying Gronwall's equality to \eqref{C424}  and using \eqref{3.22}, \eqref{3.24}, we  obtain that
$$\sup_{0\leq t\leq T}\|\nabla\rho\|_{L^4}\leq 3e^{C_5(C_6+C_7)}\|\nabla\rho_0\|_{L^4},$$
by supposing $\|\mu(\rho_0)-1\|_{L^\infty}\leq\varepsilon_{4}\triangleq \min\{6^{-3}e^{-7 C_5(C_6+C_7)}\|\nabla\rho_0\|_{L^4}^{-3},\,\varepsilon_3,\,\varepsilon_{41},\,(2C_4)^{-1}\}$, which along with \eqref{C423}, \eqref{3.22}, \eqref{3.24}, \eqref{3.39}, and \eqref{3.7} concludes that
\begin{equation}\label{nr}\notag
\sup_{0\leq t\leq T}\|\rho\|_{H^1\cap W^{1,q}}+\int_0^T\|\nabla u\|_{L^\infty}\,dt\leq C.
\end{equation}
The proof of Lemma \ref{le6} is therefore finished.
\end{proof}

Then the regularity assertions \eqref{1.7}--\eqref{1.8} follows from standard elliptic estimates and interpolations.

\begin{lemma}\label{le7}
Under the conditions of Theorem \ref{T1},suppose that
$$\|\mu(\rho_0)-1\|_{L^\infty}\leq\tilde{\varepsilon},$$
with $\tilde\varepsilon$ $(\le \varepsilon_4)$  depending on $\bar\rho$, $\mu$, $q$, and the initial values. There is a positive constant $C$ determined by $\bar\rho$, $\mu$, $q$, 
and the initial values such that 
\begin{equation}\begin{aligned}\label{3.288}
&\sup_{t\in[0,T]}\delta(t)\, e^{\sigma t}\left(\|\nabla u\|_{H^1}^2+\| P\|_{H^1}^2+\|\sqrt\rho  u_t\|_{L^2}^2\right)\\
&+\int_0^Te^{\sigma t}\left(\|\nabla u\|_{H^1}^2+\|P\|_{H^1}^2+\|\sqrt\rho  u_t\|_{L^2}^2\right)\,dt\\
&+\int_0^T \delta(t)e^{\sigma t/2}\left(\|\nabla  u\|_{W^{1,q}}^2+\| P\|_{W^{1,q}}^2+\|\nabla u_t\|_{L^2}^2\right)dt\leq C.
\end{aligned}\end{equation}
In particular, the above estimates imply that
\begin{equation}\label{CC501}
\|u-u_\infty\|_{L^5}+|u_\infty|+\|u\|_{L^\infty}\leq Ce^{-\sigma t},~~\forall t\in[1,\infty).
\end{equation}
\end{lemma}
\begin{proof}
Observe that \eqref{C202}, \eqref{C339}, \eqref{1.5},  \eqref{C225}, and \eqref{3.11} give 
\begin{equation}\label{pe}\begin{aligned}
\|P\|_{L^2}
&\le C\|\nabla\Delta^{-1}\mathrm{div}(\rho\dot u)\|_{L^1}+\|\Delta^{-1}\mathrm{div}\mathrm{div}(2\mu(\rho)Du)\|_{L^2}\\
&\le C\left(\|\rho\dot u\|_{\mathcal{H}^1(\R)}+\|\nabla u\|_{L^2}\right)\\
&\le C\left(\|\sqrt\rho\dot u\|_{L^2}+\|\nabla u\|_{L^2}\right).
\end{aligned}\end{equation}
Applying Lemma \ref{pl3}, \eqref{pe}, \eqref{3.19}, \eqref{C423}, and \eqref{3.23} shows that there is a small positive constant $\tilde\varepsilon\,(\le \varepsilon_4)$ such that if $\|\mu(\rho_0)-1\|_{L^\infty}\leq\tilde\varepsilon,$ 
\begin{equation}\label{CC408}
\begin{split}
&\|\nabla u\|_{H^1}+\|P\|_{H^1}\le C\left(\|\sqrt\rho\dot u\|_{L^2}+\|\nabla u\|_{L^2}\right),\\
&\|\nabla u\|_{W^{1,q}}+\|P\|_{W^{1,q}}\le C(1+t)^5\left(\|\sqrt\rho \dot u\|_{L^2}+\|\nabla\dot u\|_{L^2}+\|\nabla u\|_{L^2}\right).
\end{split}
\end{equation}
Therefore, from above estimates and \eqref{3.11}, we get
\begin{equation}\begin{aligned}\label{3.40}
&\sup_{t\in[0,T]}\delta(t)\,e^{\sigma t}\left(\|\nabla u\|_{H^1}^2+\|P\|_{H^1}^2\right)+\int_0^Te^{\sigma t}\left(\|\nabla u\|_{H^1}^2+\|P\|_{H^1}^2\right)\,dt\\
&+\int_0^T\delta(t)\, e^{\sigma t/2}\left(\|\nabla u\|_{W^{1,q}}^2+\|P\|_{W^{1,q}}^2\right)\,dt\leq C.
\end{aligned}\end{equation}

To estimate the time derivatives, we recall \eqref{3.7}, \eqref{3.11}, and \eqref{C425} to deduce that there is a constant 
 $u_\infty(t)$ decided by $u$ such that
\begin{equation}\label{CC405}
\|u-u_\infty(t)\|_{L^{5}}\leq C\||\nabla u|\cdot|x|^{2/5}\|_{L^2}
\leq C\|\sqrt{\rho}\dot{u}\|_{L^2}.
\end{equation}
Meanwhile, \eqref{CC402} and \eqref{CC405} also indicates that
\begin{equation*}
\begin{split}
\int_\R\rho|u_\infty(t)|^2\,dx&=\int_\R\rho|u-u_\infty(t)|^2\,dx-\int_\R\rho|u|^2\,dx,\\
&\leq C(\|\sqrt{\rho}u\|_{L^2}^2+\|\sqrt{\rho}\dot{u}\|_{L^2}^2),
\end{split}
\end{equation*}
which along with the conservation of mass $\int_\R\rho\,dx=\int_\R\rho_0\,dx$ gives
\begin{equation}\label{CC406}
|u_\infty(t)|\leq C(\|\sqrt{\rho}u\|_{L^2}+\|\sqrt{\rho}\dot{u}\|_{L^2}).  
\end{equation}
Therefore, the sobolev embedding, \eqref{CC408}, and \eqref{CC406} imply that
\begin{equation}\label{CC407}
\begin{split}
\|u\|_{L^\infty}&\leq |u_\infty(t)|+\|u-u_\infty(t)\|_{L^{\infty}}\\
&\leq|u_\infty(t)|+\|u-u_\infty(t)\|_{L^{16}}+\|\nabla u\|_{H^1}\\
&\leq C(\|\sqrt{\rho}u\|_{L^2}+\|\sqrt{\rho}\dot{u}\|_{L^2}).
\end{split}
\end{equation}
Now let us make use of \eqref{3.11}, \eqref{3.40}, and \eqref{CC407} to check that
\begin{equation*}\begin{aligned}
\|\sqrt\rho u_t\|_{L^2}
&\le \|\sqrt\rho \dot u\|_{L^2}+\|\sqrt\rho u\cdot\nabla u\|_{L^2}\\
&\le \|\sqrt\rho \dot u\|_{L^2}+\|\rho\|_{L^{\infty}}^{1/2}\|u\|_{L^{\infty}}\|\nabla u\|_{L^{2}}\\
&\le C\left(\|\nabla u\|_{L^{2}}+\|\sqrt\rho \dot u\|_{L^2}\right).
\end{aligned}\end{equation*}
Similarly, we also have
\begin{equation*}\begin{aligned}
\|\nabla u_t\|_{L^2}
&\le \|\nabla \dot u\|_{L^2}+\|\nabla u\|_{L^4}^2+\||u|\,|\nabla^2 u|\|_{L^2}\\
&\le \|\nabla \dot u\|_{L^2}+C\,\|\nabla u\|_{L^{2}}\|\nabla^2 u\|_{L^{2}}+C\,\|u\|_{L^\infty}\|\nabla^2 u\|_{L^2}\\
&\le \|\nabla \dot u\|_{L^2}+C\left(\|\nabla u\|_{L^{2}}+\|\sqrt\rho \dot u\|_{L^2}\right)^2.
\end{aligned}\end{equation*}
The above two estimates together with \eqref{3.11} ensure that
\begin{equation}\begin{aligned}\label{3.43}\notag
\sup_{t\in[0,T]}\delta (t)\,e^{\sigma t}\|\sqrt\rho  u_t\|_{L^2}^2+\int_0^T\left(e^{\sigma t}\|\sqrt\rho u_t\|_{L^2}^2+\delta(t)e^{\sigma t/2}\|\nabla u_t\|_{L^2}^2\right)\,dt\le C,
\end{aligned}\end{equation}
which combined with \eqref{3.40} gives \eqref{3.288}. Collecting \eqref{3.11} and \eqref{CC405}--\eqref{CC407} also implies \eqref{CC501}. The proof of Lemma \ref{le7} is therefore finished.
\end{proof}

\begin{remark}\label{R43}
For the case $\mu(\rho)\equiv 1$ in Theorem \ref{T2}, we can control the $L^1(0,T;L^\infty)$-norm of $\nabla u$ and obtain Lemma \ref{le7} from the elliptic estimates (see \cite[Chapter 9]{GT}) of the Stokes system with constant coefficient in $\R$:
\begin{equation}\label{C427}
\begin{cases}
-\mu_0\Delta u+\nabla P=-\rho\dot u,\\
\mathrm{div}u=0,
\end{cases}
\end{equation}
The Gagliardo-Nirenberg inequality and $W^{2,p}$ estimates of \eqref{C427} imply that
\begin{equation*}
\begin{split}
\|\nabla u\|_{L^\infty}\leq C\,\|\nabla u\|_{L^2}^{1/4}\|\nabla^2u\|_{L^3}^{3/4}\leq C\,\|\nabla u\|_{L^2}^{1/4}\|\rho\dot u\|_{L^3}^{3/4},\\
\end{split}
\end{equation*}
which along with  \eqref{3.22} and \eqref{3.24}  gives the uniform estimate on $L^1(0,T;L^\infty)$-norm of $\nabla u$. While the proof of Lemma \ref{le7} remains the same line.
\end{remark}

\section{Proof of main Theorems}\label{S5}
\quad With all a priori estimates in Section  \ref{S4} at hand, we can prove the main Theorems \ref{T1} and \ref{T2} of this paper.

{\it Proof of Theorem \ref{T1}:} 
Under the conditions of Theorem \ref{T1}, we assume further that  
\begin{equation*}
\|\mu(\rho_0)-1\|_{L^\infty}\leq\tilde\varepsilon,
\end{equation*}
with $\tilde\varepsilon$ as in Lemma \ref{le7}. Then
according to Lemma \ref{L41}, there exists constant  $T_0>0$ such that the system \eqref{C101}--\eqref{1.2} admits a unique strong solution $(\rho, u, P)$ on $\R\times(0, T_0]$, which satisfies \eqref{1.7}. 

Let $T^*$ be the largest existence time of the unique solution $(\rho,u,P)$, namely
\begin{equation*}
T^*\triangleq \sup\big\{T>0\,\big|\,(\rho,u,P)
~\mathrm{solves}~\mbox{\eqref{C101}--\eqref{1.2}}~\mathrm{in}~\R\times[0,T]~\mathrm{and \, satisfies}~\eqref{1.7}\big\}.
\end{equation*}
Note that $T^*$ is well defined, since $T_0$ belongs to the set given above.
Consequently, for any $0< \tau<T< T^{*}$, one deduces from \eqref{1.7} that 
\begin{equation}\label{C501}
\rho u,\,\nabla u\in L^{\infty}(\tau,T;H^1)\cap H^1(\tau,T;L^2) \subset C([\tau,T];L^2\cap L^q).
\end{equation} 
In addition, according to \cite[Lemma 2.3]{lions}, the transport equation \eqref{C101}$_1$ combined with \eqref{3.19} is sufficient to guarantee that
\begin{equation}\label{C502}
\rho \in C([0,T]; H^1\cap W^{1,q}).
\end{equation}

Now, we claim
$T^*=\infty$. 
Otherwise, if $T^*<\infty$, we can utilize Lemmas \ref{le1} and \ref{le5}--\ref{le7} along with \eqref{C501}--\eqref{C502} to define
$$(\rho^*,u^*)(x)\triangleq(\rho, u)(x,T^*)=\lim_{t\rightarrow T^*}(\rho, u)(x,t),$$
which satisfies the initial conditions \eqref{1.3}, that is,
\begin{equation}\notag 
\begin{split}
0\leq\rho^*\leq \bar\rho,\quad \rho^*\,(1+|x|^2)\in L^1,\quad \rho^\ast\in H^1\cap W^{1,q}, \quad \sqrt{\rho^*}u^*\in L^2,\quad u^*\in D^1_{0,0}.
\end{split}
\end{equation}  
Therefore, one can take $(\rho^*,\,\rho^*u^*)$ as the initial data and apply
Lemma \ref{L41} to extend $(\rho,u,P)$ beyond $T^*$ uniquely,
which contradicts the definition of $T^*$. 
Thus we must have $T^*=\infty$, moreover the uniform bounds \eqref{C102} and exponential decay \eqref{1.8}--\eqref{CC96} follow from Lemmas \ref{le5}--\ref{le7}. The proof of Theorem \ref{T1} is completed.  
\thatsall
\\
\vspace{0.1 cm}

{\it Proof of Theorem \ref{T2}:} When $\mu(\rho)\equiv 1$, the global well-posedness of weak solutions to the system \eqref{C101}--\eqref{1.2} follows from \cite{DM2,HSWZ}, hence we concentrate on the large time behaviors \eqref{C105}--\eqref{CC94}.

The cancellation condition $\rho\dot u$ and conservation of momentum remain true according to Lemmas \ref{LL33} and \ref{LL34}.
Consequently, the a priori estimates and exponential decay in Lemmas \ref{le1}--\ref{le5} and \ref{le7} are still valid as mentioned in Remark \ref{R43}, and then the large time behaviors \eqref{C105}--\eqref{CC95} are valid.  

Let us illustrate the uniform-in-time regularity of flow map. We still write the momentum equation \eqref{C101}$_2$  as a Stokes system with constant coefficient in $\R$:
\begin{equation}\label{C503}
\begin{cases}
-\mu_0\Delta u+\nabla P=-\rho\dot u,\\
\mathrm{div}u=0.
\end{cases}
\end{equation}
Applying the $W^{2,p}$ estimates (see Lemma \ref{pl3} and \cite[Chapter 9]{GT}) to \eqref{C503}, Sobolev  inequalities, \eqref{C105}, \eqref{C425}, and \eqref{CC407} yields that for $\beta=1/4$ and $0<\gamma<1$, there is a  positive constant $C$ depending only on $\gamma$ such that
\begin{equation*}
\begin{split}
\|u\|_{L^\infty}&\leq C\big(|u_\infty(t)|+\|u-u_\infty\|_{L^{4/\beta}}+\|\nabla u\|_{L^3}\big)\leq C\big(\|\nabla u\|_{L^2}+\|\sqrt{\rho}\dot u\|_{L^2}\big), \\
\|\nabla u\|_{C^\gamma}&\leq C\,\|\nabla u\|_{W^{1,2/(1-\gamma)}}\leq C\big(\|\nabla u\|_{L^2}+\|\sqrt{\rho}\dot u\|_{L^2}+\|\rho\dot u\|_{L^{2/(1-\gamma)}}\big),
\end{split}
\end{equation*}
which along with similar estimates in \eqref{3.22} and \eqref{3.24} ensure us to find a constant $C$ determined by $\gamma$, $\bar\rho$, $\|\rho_0\|_{L^1}$, $\|\rho_0\,|x|^2\|_{L^1}$, $\|\sqrt{\rho_0}u_0\|_{L^2}$, and $\|\nabla u_0\|_{L^2}$ such that for any $T\in(0,\infty),$
\begin{equation}\label{hn}
\int_0^T\left(\|u\|_{L^\infty}+\|\nabla u\|_{C^\gamma}\right)\,dt\le C.
\end{equation}
This uniform estimate  guarantees that the flow map in \eqref{C109} is well defined. Moreover, using the flow map and Gronwall's inequality,  
we obtain that for any  $T\in (0,\infty)$, 
\begin{equation}\notag
\begin{split}
\exp\left\{-\int_0^T\|\nabla u\|_{L^\infty}\,dt\right\}\leq\frac{|X_T(x)-X_T(y)|}{|x-y|}\leq \exp\left\{\int_0^T\|\nabla u\|_{L^\infty}\,dt\right\},
\end{split}
\end{equation}
\begin{equation}\notag
\begin{split}
\|\nabla_xX_T(x)\|_{C^{\gamma}}
&\leq \|\nabla_xX_0(x)\|_{C^{\gamma}}\exp\left\{\int_0^T\|\nabla u\|_{L^\infty}\,dt\right\}\\
&\quad+\int_0^T\|\nabla u\|_{C^{\gamma}}\|\nabla_xX_t\|_{L^\infty}^{1+\gamma}\,\exp\left\{\int_t^T\|\nabla u\|_{L^\infty}\,ds\right\}\,dt,
\end{split}
\end{equation} 
which along with \eqref{hn} lead to \eqref{C107}--\eqref{C118}.

In particular, \eqref{CC95} ensures that 
$$X_\infty(x)=
\lim_{t\rightarrow\infty}X_t(x)=x+\lim_{t\rightarrow\infty}
\int_0^tu(X_s(x),s)\,ds,$$
is well defined, and \eqref{C107} along with $\mathrm{div}u=0$ also implies that $X_\infty$ is a measure preserving $C^{1,\gamma}$ diffeomorphism. Now it is direct to check \eqref{CC94} that
\begin{equation*}
\|X_t(\cdot,t)-X_\infty(\cdot)\|_{L^\infty}
\leq\int_t^{\infty}\|u\|_{L^\infty}\,ds\leq Ce^{-\sigma t}.
\end{equation*}
The proof of Theorem \ref{T2} is therefore complete.
\thatsall

\section*{Statements and Declarations}
{\bf Conflict of interest statement}
On behalf of all authors, the corresponding author states that there is no conflict of interest. \\
{\bf  Data availability statement} 
No data were used for this research.

\section*{Acknowledgements} 
Xinyu Fan is supported by "the Fundamental Research Funds for the Central Universities", Nankai University. Boqiang L\"{u} is supported by the National Natural Science Foundation of China (No.  12371219), the Double-Thousand Plan of Jiangxi Province (No. jxsq2023201115);

\end{document}